\documentclass[a4paper]{article}
\usepackage{geometry}
\usepackage[english]{babel}
\usepackage{enumerate,amsmath,amsthm,amssymb,amsfonts,mathrsfs}
\usepackage[longend,ruled,linesnumbered]{algorithm2e}
\usepackage{fancyhdr}
\usepackage{array,listings,color,graphicx}
\usepackage{hyperref}
\usepackage[numbers, sort]{natbib}
\usepackage[T1]{fontenc}
\usepackage[utf8]{inputenc}
\usepackage{authblk}
\hypersetup{hypertex=true,colorlinks=true,
linkcolor=blue,
anchorcolor=red,
citecolor=red}

\newcommand{\tcb}[1]{\textcolor{blue}{#1}}

\numberwithin{equation}{section}  

\newtheorem{lemma}{Lemma}[section]
\newtheorem{proposition}{Proposition}[section]
\newtheorem{corollary}{Corollary}[section]
\newtheorem{remark}{Remark}[section]
\newtheorem{assumption}{Assumption}[section]
\newtheorem{theorem}{Theorem}[section]
\newtheorem{definition}{Definition}[section]
\renewcommand{\tcb}[1]{#1}

\def\R{\mathbb{R}}
\def\P{\mathbb{P}}
\def\E{\mathbb{E}}

\def\cS{\mathcal{S}}
\def\cH{\mathcal{H}}

\def\tr{\mathrm{tr}}

\begin{document}
	\title{Estimates of the numerical density for stochastic differential equations with multiplicative noise}
\author[a,b]{Lei Li\thanks{E-mail: leili2010@sjtu.edu.cn}}
\author[a]{Mengchao Wang\thanks{E-mail: mc666@sjtu.edu.cn} }
\author[a]{Yuliang Wang\thanks{E-mail: YuliangWang$\_$math@sjtu.edu.cn}}
\affil[a]{School of Mathematical Sciences, Institute of Natural Sciences, MOE-LSC, Shanghai Jiao Tong University, Shanghai, 200240, P.R.China.}
\affil[b]{Shanghai Artificial Intelligence Laboratory}
\date{}
 
	\maketitle

\begin{abstract}
We investigate the estimates of the density
for the traditional Euler-Maruyama discretization of stochastic differential equations (SDEs) with multiplicative noise.
Our estimates focus on two key aspects: (1) the $L^p$-upper bounds for derivatives of the logarithmic numerical density, (2) the sharp error order of the Euler scheme under the relative entropy (or Kullback-Leibler divergence).  
For the first aspect, we present estimates for the first-order and second-order derivatives of the logarithmic numerical density. The key technique is to adopt the Malliavin calculus to derive expressions of the derivatives of the logarithmic Green's function and to obtain an estimate for the inverse Malliavin matrix.
Moreover, for the relative entropy error, we obtain a bound that is second order in time step, which then naturally leads to first-order error bounds under the total variation distance and Wasserstein distances. Compared with the usual weak error estimate for SDEs, such estimate can give an error bound for the worst case of a family of test functions instead of one test function. 
\end{abstract}


\section{Introduction}
Stochastic Differential Equations (SDEs) provide a robust mathematical framework for modeling a wide range of dynamic systems influenced by random processes, including open physical systems \cite{van1992stochastic,li2017fractional}, stock prices of financial derivatives \cite{lamberton2011introduction,klebaner2012introduction}, and even in training process of neural networks \cite{ben2022high,hu2019diffusion}.
In this paper, we are concerned with the SDEs with multiplicative noise of the following general form in It\^o's sense \cite{oksendal2003stochastic,karatzas2014brownian} 
\begin{equation}\label{eq:SDE_multiplicative}
	d X_{t} = b ( t, X_t) d t + \sigma (t, X_t ) d W_{t}, 
	\quad t > 0,\quad X_0\sim \rho_0,
\end{equation}
where $W_t $ is a given $m$-dimensional Wiener process (standard Brownian motion) on a filtered probability space
$ (\Omega, \mathscr{ F }, \mathbb{P}, \{\mathscr{ F }_t\}_{t \geq 0})$, and the function 
$ b : \mathbb{R}_{+}\times\mathbb{R}^{d} \to \mathbb{R}^d$
is the drift term,  the function $ \sigma : \mathbb{R}_{+}\times\mathbb{R}^{d} \to \mathbb{R}^{d\times m}$ is the diffusion term. 
The initial value $X_0$ is drawn from a given law $\rho_0$.

Numerical discretization and simulations of  SDEs are important in real applications and have been studied for a long time
\cite{kloeden1992stochastic,Milstein2004Stochastic,Jentzen2011Taylor}. The analysis for behaviors of numerical densities associated with these schemes plays a fundamental role in the further study of numerical SDEs, and it is often treated by the Malliavin calculus 
\cite{hong2024density,cui2022density,chen2023accelerated}.
For example, in \cite{cui2022density}, the authors studied the density of the numerical solutions of the splitting averaged vector field (AVF) scheme for the stochastic Langevin equation using Malliavin calculus. They established the existence of the density and showed that the convergence rate of the numerical density aligns with the strong convergence rate of the numerical scheme.
Unfortunately, to the best of our knowledge, compared with other related results, rigorous estimates for numerical densities remain relatively scarce in the existing literature, particularly in the case of multiplicative noise (i.e. the function $\sigma$ in \eqref{eq:SDE_multiplicative} depends on $X_t$). This gap in the literature motivates us to conduct a state-of-the-art analysis of numerical densities, and the details will be discussed in the first part of this paper later. Specifically, we establish a sequence of estimates for logarithmic numerical densities, mainly via Malliavin calculus. These estimates for numerical densities enable researchers to conduct advanced convergence analysis of numerical schemes for SDEs.
When it comes to convergence analysis of some numerical scheme, typically, one uses the so-called strong convergence and weak convergence to study the numerical methods. The strong convergence focuses on the approximation ability of the stochastic trajectories while the weak convergence focuses on the convergence of the laws \cite{kloeden1992stochastic,Milstein2004Stochastic}.
Traditionally, the weak convergence focuses on the convergence under the weak star topology against smooth functions and is often estimated by fixing an arbitrary test function and thus the convergence error depends on the test function \cite{Milstein2004Stochastic}. Such error estimate is not suitable if we are concerned with statistics for a family of nonsmooth test functions. For example, the Wasserstein-1 distance for two probability measures on $\R^d$ is given for a family of Lipschitz test functions \cite{villani2009optimal} 
\begin{gather}
W_1(\mu, \nu)=\sup_{\varphi\in \mathrm{Lip}_1}\left|\int_{\R^d} \varphi\, d(\mu-\nu)\right|,
\end{gather}
where $\mathrm{Lip}_1$ represents the class of functions that Lipschitz continuous with constant $1$.
Therefore, instead of beginning with a fixed test function, one needs to estimate the error bound for the numerical densities and the densities for the time-continuous SDEs directly, under a certain strong enough metric. The most frequently used metrics include the total variation norm, and the Wasserstein distances \cite{villani2009optimal,santambrogio2015optimal}. Remarkably, using the strong error for trajectories, it is possible to obtain some error bound under these distances. See, for instance, recent convergence results on Euler schemes and splitting schemes for (underdamped) Langevin equations under Wasserstein distances \cite{leimkuhler2024contraction,leimkuhler2023contraction}. However, the problem is that the strong error for the Euler scheme is only half-order with multiplicative noise so the estimate is not sharp. To address this issue, in the second part of this work, we directly study properties of numerical densities rather than the trajectories, and then obtain a sharp error bound. As a brief overview, our analysis is centered around these two main aspects:
\begin{itemize}
	\item \textbf{The $L^p$ upper bounds for derivatives of the logarithmic numerical density:} 
	 The integrability of $p$-th moment of logarithmic numerical density is crucial but challenging in numerical analysis, especially in the multiplicative noise scenario. To address this, we employ Malliavin calculus \cite{nualart2018introduction,bally2010introduction} to derive expressions for the derivatives of the logarithmic Green's function. By making use of Malliavin matrices and various Malliavin derivatives, we obtain integrated versions of the desired estimates for the logarithmic density. The estimate of this kind is crucial for deriving sharp error bounds under both additive \cite{mou2022improved} and multiplicative settings.
	
	\item \textbf{The sharp error order of the Euler scheme under the relative entropy:} We extend the relative entropy estimate from the additive noise case \cite{mou2022improved} to the multiplicative noise case. In detail, we establish an $O(h^2)$
	error bound for the Euler-Maruyama scheme of SDEs with multiplicative noises in terms of the relative entropy. This sharp error estimate is crucial for understanding the convergence of numerical approximations of SDEs, particularly under metrics such as the total variation norm and Wasserstein distances. Existing works typically focus on weak convergence estimates or strong convergence results with less precision when multiplicative noise is involved. Our result improves upon these estimates, providing a sharp error bound for the Euler scheme under relative entropy.
\end{itemize}

Throughout this paper, we consider Euler's method (Euler-Maruyama scheme) defined as follows:
\begin{equation}\label{eq:Euler-method}
X_{ k+1 }^{h} =   X_{ k  }^{h}+ h b ( t_{k}, X_{ k  }^{h}  )+  \sigma( t_{k}, X_{ k  }^{h} ) (W_{t_{k+1}} - W_{t_{k}}),
\end{equation}
where $h>0$ indicates the time step, $t_k=kh$,  $X_k^h$ indicates the corresponding numerical solution at $t_k$, and $W_{t_k}$ is the evaluation of the Wiener process at $t_k$. We will also assume that 
\begin{gather}
X_0^h\sim \rho_0.
\end{gather}
For weak schemes, one may replace $W_{t_{k+1}}-W_{t_k}$ by independent normal variables, irrelevant to the original Wiener process. Since the laws will stay unchanged whether we use $W_{t_{k+1}} - W_{t_{k}}$ or independent normal variables, we will stick to \eqref{eq:Euler-method} in the remaining part of the paper.

Our tool to study \eqref{eq:Euler-method} is the relative entropy (or Kullback-Leibler divergence (KL divergence)), which is a widely used quantity in statistics and data sciences \cite{kullback1951information,mackay2003information}. Besides, the relative entropy has been used to study the 
interacting particle systems \cite{arous1999increasing,jabin2018quantitative,lacker2023hierarchies,huang2024mean}. Recently, it has been used to treat the propagation of chaos for particle systems with singular kernels \cite{jabin2018quantitative}, for Landau type equations \cite{carrillo2024relative,du2024collision}, and for stationary problems \cite{arous1999increasing,li2023solving}. 
When it comes to the numerical approximations of SDEs, the relative entropy has been used to study the error bound 
of Langevin Monte Carlo
\cite{cheng2018convergence,dalalyan2019user}. The Langevin Monte Carlo is based on the Langevin equation, which is an SDE with additive noise. Recently, some sharp error bounds have been obtained for the Euler method of the Langevin equation \cite{mou2022improved} and its random batch version \cite{li2022sharp} using relative entropy. In particular, the bound has been improved to $O(h^2)$ in \cite{mou2022improved,li2022sharp} for the relative entropy where $h$ is the constant time step, and consequently
first order $O(h)$ under the total variation and Wasserstein distances.    There are fewer results about the sharp error estimates of the densities for numerical method of SDEs with multiplicative noises. Some nonasymptotic bounds for the convergence to invariant measure under the Wasserstein distance have been obtained for the Euler scheme with multiplicative noise in \cite{pages2023unadjusted}, and the convergence of simulated annealing type algorithms have been studied in \cite{bras2024convergence,bras2023convergence} for multiplicative noise. 
Recently,  some sharp error bounds under the total variation and Wasserstein-1 distances have been obtained in \cite{zhangwang2024total} for the multiplicative noise and non-globally Lipschitz coefficients. 
The sharp estimates under relative entropy and Wasserstein-2 distance are still unavailable to the best of our knowledge.

In this paper, we build upon these results and derive some more careful estimates for the Euler-Maruyama scheme with multiplicative noise.  By leveraging Malliavin calculus and carefully estimating the inverse Malliavin matrix, we provide upper bounds for the derivatives of the logarithmic numerical density. These estimates are crucial for establishing the sharp convergence result under the relative entropy, where we show that the error bound for the relative entropy is 
$O(h^2)$, improving upon most previous results.



The rest of the paper is organized as follows. In Section \ref{sec:prelim}, we give a brief introduction to the relative entropy and Malliavin calculus and introduce some necessary notations. In Section \ref{sec:assumption}, we state the main assumptions used throughout this paper. 
Section \ref{sec:nablalog} develops the foundational estimates of numerical density that constitute the first core achievement of this study. In direct theoretical progression, Section \ref{sec:mainresult} articulates an equally significant breakthrough by establishing the relative entropy bound through systematically employing the density estimates, thereby completing the dual theoretical pillars of this work.
In particular, we apply the Malliavin calculus to find the expressions for the gradients of the logarithmic Green's function. By an estimate of the inverse Malliavin matrix, we show that the $x$ derivative and $y$ derivative can cancel in the leading order which allows us to get the estimates of the numerical density. 
\tcb{Section \ref{sec:experiments} numerically validates the second-order convergence in relative entropy for the Euler-Maruyama scheme (Theorem~\ref{thm:errentropy}) through two benchmark cases: geometric Brownian motion and a $2$-dimensional Ginzburg-Landau model.}
In Appendix \ref{app:pointwise}, we provide a pointwise estimate for 1D time continuous SDEs, also mainly via the technique of Malliavin calculus.

\section{Notations and preliminaries}\label{sec:prelim}

In this section, we collect some basic notations and tools for the rest of the paper.
In particular, we first give an introduction to the relative entropy and then to the basics of Malliavin calculus.
Last, we gather some notations that will be used in this paper.

\subsection{The relative entropy}\label{subsec:relativeentropy}

The relative entropy (KL divergence) between two probability measures $\mu$ and $\nu$ on Polish space $E$ is defined as
\begin{gather}
\cH(\mu\mid \nu) = \left\{
\begin{aligned}
  &\int_E \log \frac{\mathrm{d} \mu}{\mathrm{d}\nu} \mathrm{d}\mu, & \text{if} ~ \mu \ll \nu,\\
 & \infty, & \text{else},
\end{aligned}\right.
\end{gather}
where $\frac{\mathrm{d} \mu}{\mathrm{d}\nu}$ denotes the Radon-Nikodym derivative of $\mu$ with respect to $\nu$. 

Note that the relative entropy is non-negative by Jensen's inequality and achieves zero only if $\mu= \nu$. Moreover, it is a convex functional with respect to either argument with respect to the weak convergence topology of the probability measures. The relative entropy has a variational characterization
\begin{gather}
\cH(\mu \mid \nu)=\sup_{\varphi\in C_b(E)}\left(\int_E\varphi d\mu-\int_E e^{\varphi} d\nu\right)+1,
\end{gather}
where $C_b(E)$ indicates the class of bounded continuous functions.

The Pinsker's inequality \cite{pinsker1964information,bolley2005weighted} claims that the total variation norm (and some weighted versions) can be controlled by the square root of the relative entropy
\begin{gather}
TV(\mu, \nu)\le \sqrt{2\cH(\mu\mid \nu)}.
\end{gather}
This inequality holds without any restrictions on the measures involved.

Another class of important functional inequalities is the transportation inequalities \cite{talagrand1996transportation,bobkov2001hypercontractivity}. The Talagrand transportation inequality gives the control for $W_2(\mu,\nu)$ if $\nu$ satisfies a log-Sobolev inequality with constant $\lambda$ \cite{talagrand1991new,otto2000generalization}:
\begin{equation}\label{eq:talaW2}
W_2(\mu,\nu) \leq \sqrt{\frac{2}{\lambda} \cH(\mu |\nu)}.
\end{equation}
Here, we recall the log-Sobolev inequality. Consider the entropy of the square of a Lipschitz function $f$ with respect to $\nu$, defined as:
\begin{equation}
	\operatorname{Ent}_{\tcb{\nu}}(f^2) = \int_{\R^d} f^2 \log(f^2) d\tcb{\nu} - \left( \int_{\R^d} f^2 d\tcb{\nu} \right) \log \left( \int_{\R^d} f^2 d\tcb{\nu} \right).
\end{equation}
We say $\nu$ satisfies the log-Sobolev inequality with constant $\lambda$ if for all such $f$, one has 
\begin{equation}
	\operatorname{Ent}_{\nu}(f^2) \leq \frac{1}{2\lambda} \int_{\R^d} |\nabla f|^2 d\nu.
\end{equation}

Moreover, using the weighted Csisz\'ar-Kullback-Pinsker inequality \cite{bolley2005weighted}, one has
\begin{equation}
W_1(\mu,\nu) \leq C_{\nu}\sqrt{ \cH(\mu |\nu)},
\end{equation}
for any $\nu$ satisfying the following tail behavior:
\begin{equation}\label{eq:W1condition}
C_\nu:=2 \inf _{\alpha>0}\left(\frac{1}{2 \alpha}\left(1+\log \int_{\R^d} e^{\alpha|x|^2} d \nu(x)\right)\right)^{\frac{1}{2}}<+\infty.
\end{equation}

\subsection{Basics of Malliavin calculus}\label{subsec:malliavin}

Let us introduce some notations and basic definitions of Malliavin Calculus \cite{nualart2018introduction}. 
Consider the given $m$-dimensional Brownian motion $W_t$ and the associated probability space $X:=(\Omega, \hat{\mathscr{F}}, \P, \hat{\mathscr{F}}_t)$, where $\hat{\mathscr{F}}_t$ is the filtration generated by $W_t$ and $\hat{\mathscr{F}}$ is the $\sigma$-algebra generated by $W_t$. (Note that $\hat{\mathscr{F}}$ and $\hat{\mathscr{F}}_t$ could be smaller than the given $\mathscr{F}$ and $\mathscr{F}_t$ because the latter could contain other information like the randomness in the initial values.)

One would call the distributional derivative of the Brownian motion $\eta_t=\dot{W}_t$ the white noise.
The Malliavin derivative $D: L_{\mathrm{loc}}^1(\Omega; \R)\to L_{\mathrm{loc}}^1( \Omega; L^2([0,\infty), \R^m))$ acting on a measurable $F: \Omega \to \R$ is understood as the variation of $F$ with respect to white noise. The output is a measurable random variable taking values in $L^2([0,\infty); \R^m)$. Here, $F$ is measurable with respect to $\hat{\mathscr{F}}$ (not the given $\mathscr{F}$), meaning that it is in fact a functional of $W_t$. Let us denote $H=L^2([0,\infty); \R^m)$ for the convenience.
Letting $h\in H$, one denotes
\begin{gather}
W(h):=\int_0^{\infty}h_s dW_s,
\end{gather}
the It\^o's integral of $h$. To rigorously define $D$, one may start with a class of special random variables
\begin{gather}\label{eq:cSdef}
F\in \cS:=\{f\left(W\left(h_1\right), \ldots, W\left(h_n\right)\right): f\in C_p^{\infty}\left(\mathbb{R}^{n}\right), n\ge 1, h_i\in H, 1\le i\le n \},
\end{gather}
where $C_p^{\infty}\left(\mathbb{R}^n\right)$ is the set of all infinitely differentiable functions $f: \mathbb{R}^n \rightarrow \mathbb{R}$ such that all the partial derivatives have at most polynomial growth. Note that the set $\mathcal{S}$ contains such variables for all possible $n\ge 1$.

\begin{definition}
Consider a \tcb{random} variable $F=f\left(W\left(h_1\right), \ldots, W\left(h_n\right)\right)\in \mathcal{S}$ for some $n\ge 1$. For any $t\ge 0$, the Malliavin derivative at $t$ is defined as
\begin{gather}
D_t F=\sum_{j=1}^n \frac{\partial f}{\partial x_j}\left(W\left(h_1\right), \ldots, W\left(h_n\right)\right) h_j(t) .
\end{gather}
Moreover, $DF:=(D_tF)_{t\ge 0}$ is the process in $L^2( \Omega; H)$.
\end{definition}
For general $F\in L^p(\Omega; \R)$ ($p\ge 1$), the Malliavin derivative can be defined by approximating $F$ using the variables in $\cS$
through a density argument. The set of $F\in L^p$ such that $DF\in L^p(\Omega; H)$ is denoted by $\mathbb{D}^{1,p}$. \tcb{The definitions here work also for the case if we only consider the integrals up to time $T$, by considering $h_i\in C_c([0, T])$ (also note that $C_c([0,T])$ is dense in $L^2([0,T])$).}  The Malliavin derivatives can extend to various Sobolev spaces and see \cite{nualart2018introduction} for more details. 
Moreover, one can similarly define repeated Malliavin derivatives like $DDF:=D(DF)$ etc.  Below, we will mostly focus on related concepts for $DF$ with $p=2$. The general case is just similar and one may refer to \cite{nualart2018introduction}.

Using the duality, one can define the following divergence operator $\delta$.
\begin{definition}
	The domain of the divergence operator $\mathrm{Dom}~\delta$ in $L^2(\Omega)$ is the set of processes $u \in L^2\left(\Omega; H\right)$ such that there exists $\delta(u) \in L^2(\Omega)$ satisfying the duality relationship
\begin{gather}\label{eq:byparts}
\E\left(\langle D F, u\rangle_{H}\right)=\E(F \delta(u) ),
\end{gather}
	for any $F \in \mathbb{D}^{1,2}$. Here, $\langle f, g\rangle_H=\int_0^{\infty}f\cdot gdt$ is the inner product in $L^2([0,\infty); \R^m)$.
\end{definition}
The definition \eqref{eq:byparts} is in fact an integration by parts formula.
\tcb{Recalling $\mathcal{S}$ defined in \eqref{eq:cSdef}, we further define the class $\cS_H$ by:}
\begin{equation}
    \tcb{\mathcal{S}_H := \left\{\sum_{j=1}^n F_j h_j(t) : F_j \in \mathcal{S}, h_j \in H \right\}.}
\end{equation}
The divergence operator $\delta$ has many important properties and we list some here.
\begin{itemize}
\item The divergence operator $\delta$ is a closed linear operator  in the sense that if the sequence $u_n \in \mathcal{S}_H$ satisfies
\[
u_n \xrightarrow{L^2(\Omega ; H)} u \quad \text { and } \delta\left(u_n\right) \xrightarrow{L^2(\Omega)} G,
\]
as $n \rightarrow \infty$, then $u$  belongs to  $\mathrm{Dom}~\delta$ and $\delta(u)=G$.

\item $\tcb{\mathbb{E}}(\delta(u))=0$. In particular, any $L^2$ process $u$ that is adapted to the filtration $\hat{\mathscr{F}}_t$ is in the domain of $\delta$ and it holds that
\begin{gather}\label{eq:skorokhod}
\delta(u)=\int_0^{\infty} u_s dW_s.
\end{gather}
\end{itemize}
The property \eqref{eq:skorokhod} and the integration by parts formula \eqref{eq:byparts} play the basic roles for Malliavin calculus in applications, when studying random variables that are functionals of Brownian motions.
If $u$ is not adapted, $\delta(u)$ is called the Skorohod integral of $u$.

The following are some basic properties in Malliavin calculus which we will use frequently later (the generalization to $F\in L^2(\Omega; \R^d)$ is just similar). \tcb{We refer readers to \cite[Section 3]{nualart2018introduction} for more details.}
\begin{proposition}
\begin{enumerate}
\item(Chain rule for Malliavin derivative)	Let $\varphi: \mathbb{R} \rightarrow \mathbb{R}$ be a continuous differentiable function such that $\left|  \varphi'(x)\right| \leq C\left(1+|x|^\alpha\right)$ for some $\alpha \geq 0$. Let $F \in \mathbb{D}^{1, p}$ for some $p \geq \alpha+1$. Then, $\varphi(F)$ belongs to $\mathbb{D}^{1, q}$, where $q=p /(\alpha+1)$, and
\begin{gather}
D(\varphi(F))=  \varphi'(F) D F.
\end{gather}

\item (Product rule for Malliavin derivative) Let $F_1$, $F_2 \in \mathbb{D}^{1,2}$ such that $F_1F_2\in L^2(\Omega)$ and $F_1DF_2$, $F_2DF_1\in L^2(\Omega; H)$, then $F_1F_2 \in \mathbb{D}^{1,2}$ and 
\begin{equation}
    D(F_1F_2) = F_1DF_2 + F_2 DF_1.
\end{equation}
\item (Product rule for divergence)	Let $G \in \mathbb{D}^{1,2}$ be a real-valued random variable such that $\left.G u \in L^2\left(\Omega, \tcb{L^2([0, \infty))}\right)\right)$, then
\begin{gather}\label{eq:divergenceproduct}
\delta(G u)=G \delta(u)- \left\langle D G, u \right\rangle _{H},
\end{gather}
where we suppose that the right-hand side is integrable.
\end{enumerate}
\end{proposition}

\subsection{Some other notations}

We let $| x |$ denote the Euclidean norm of a vector $x \in \mathbb{R}$. For a matrix $M$ we let $| M |$ denote its spectral norm. 
For a tensor $T$ of order $r\ge 3$, $|T|:=\sqrt{\sum_{\alpha_1 \cdots\alpha_r} T_{\alpha_1,\cdots, \alpha_r}^2}$.
For a multi-index $\alpha=\left(\alpha_1, \ldots, \alpha_d\right) \in\mathbb{N}^d$, we denote $|\alpha|:=\sum_{i=1}^d \alpha_i$. Also,
\begin{gather}
\partial^{\alpha}:=\prod_{i=1}^d\frac{\partial^{\alpha_i}}{\partial x_i^{\alpha_i}}.
\end{gather}
\tcb{For two matrices $A$ and $B$, the notation $A \succeq B$ means $B-A$ is semi positive definite. Moreover, for two vectors $x$, $y \in \mathbb{R}^d$, $x \otimes y$ is a $d \times d$ matrix with $ij$-th element being $x_i y_j$.}

\section{Assumptions}\label{sec:assumption}

Denote $\Lambda$ the diffusion matrix
\begin{gather}
\Lambda (t, x) = \sigma (t, x) \sigma^{T} (t, x).
\end{gather}
Throughout this paper, we assume the following conditions:

\begin{assumption}\label{ass:condition-drift-diffusion}

\begin{enumerate}[(a)]
\item The drift coefficient $ b \in C^3
( \R_+ \times \mathbb{R}^{d} ; \mathbb{R}^d)$ satisfies Lipschitz continuity and smoothness in the sense that
\begin{equation}\label{eq:b_timelip}
\begin{aligned}
&|b(t,x_1) - b(t,x_2)| \leq L |x_1-x_2|,\quad \forall t\geq 0,\\
	&|b(t_1,x) - b(t_2,x)| \leq L |t_2-t_1|(1 + |x|^q), \\
& \sup_{t\geq 0}| \partial_x^{\alpha} b(t,x)| \leq C(1 + |x|^{q}), \quad 
     2\le |\alpha|\le 3,
\end{aligned}
\end{equation}
for some constant $L, C > 0 $ and $ q \ge 0$.


\item The diffusion coefficient $\sigma\in C^3(\R_+\times \R^d; \R^{d\times m})$ satisfies Lipschitz continuity and smoothness in the sense that
\begin{equation}\label{eq:sigmaspace}
\begin{aligned}
    &|\sigma(t,x_1) - \sigma(t,x_2)| \leq L |x_1 - x_2|,\quad \forall t \geq 0, \\
    & \sup_{t\geq 0}|\partial_x^{\alpha}\sigma(t, x)|<C(1 + |x|^q), \quad 
     2\le |\alpha|\le 3,
\end{aligned}
\end{equation}
for some constant $L, C > 0 $ and $ q \ge 0$.
Moreover, there exists $\kappa>0$ such that  
\begin{gather}
 \Lambda(t, x) \succeq \kappa I, \quad \forall (t, x)\in \R_+\times \R^{d},
\end{gather}
\begin{equation}\label{eq:Lam_timelip}
	|\Lambda(t_1,x) - \Lambda(t_2,x)| +|\nabla_{x}\cdot \Lambda(t_1,x) - \nabla_{x}\cdot \Lambda (t_2,x)| \leq L |t_2-t_1|(1 + |x|^q).
\end{equation}
\end{enumerate}
\end{assumption}

Note that we are not assuming the boundedness of $b$. Since $b(t, 0)$ is bounded on $[0, T]$ for any $T>0$, $b(t, x)$ therefore has linear growth on finite time interval. Hence, the moment of $b(t, X)$ is reduced to the moment of $X$. Moreover, we are not assuming any growth conditions for $\partial^\alpha_x \sigma$ for $|\alpha| = 1$, because the Lipschitz condition in \eqref{eq:sigmaspace} already indicates the linear growth of the first order derivative.
\tcb{A typical example of $b$ and $\sigma$ satisfying Assumption \ref{ass:condition-drift-diffusion} above is: $b(t,x) = -x$, $\sigma(t,x) = \sqrt{1 + x^2}$ for $x \in \mathbb{R}$.}
\tcb{Also note that uniform ellipticity condition ($\Lambda(t, x) > \kappa I$, $\kappa > 0$) is necessary within our proof framework for bounding the inverse of the Malliavin matrix in Section \ref{subsec:inverseMmatrix} below.}
Below, we give some assumptions on the initial density. 
\begin{assumption}\label{ass:Uniform bound}
The initial value $X_0$ has a density function $\rho_{0}(x)$ that satisfies the following conditions
\begin{enumerate}[(a)]
\item For $p_1$ sufficiently large ($p_1 \ge C (1 + q) $ for some universal constant $ C $ and $ q $ in Assumption \ref{ass:condition-drift-diffusion}), the initial value has finite $p_1$-moment
\begin{gather}
\E|X_0|^{p_1} =\int_{\R^d} |x|^{p_1} \rho_0(x)\,dx<\infty.
\end{gather}

\item For some $ p_2 > 4$ and $p_3>2$, the derivatives of the logarithmic density are integrable with respect to $\rho_0$ in the following sense
\begin{multline}
\mathbb{E} \Big[ \big| \nabla_{x} \log \rho_0(X_{0} )
\big|^{ p_2} \Big] +\mathbb{E} \Big[ \big| \nabla_{x}^{2} \log \rho_0(X_{0} )\big|^{ p_3} \Big] \\
=\int_{\R^d}\big| \nabla_{x} \log \rho_0(x)\big|^{ p_2}
\rho_0(x)\,dx+\int_{\R^d}\big| \nabla_{x}^2 \log \rho_0(x)\big|^{p_3}\rho_0(x)\,dx<\infty.
\end{multline}
\end{enumerate}
\end{assumption}

With the assumed conditions, the following result regarding the moments is standard (see \cite[\tcb{Lemma 6.1 or Lemma 7.1}]{mao2007stochastic} for a similar proof), and we omit the proof.
\begin{lemma}\label{lem:bound-moment}
Suppose that Assumption  \ref{ass:condition-drift-diffusion}. Then for any $q>1$, there is a constant $C(q, T)>0$ that is independent of $h$ such that the interpolated process \eqref{eq:continuous-Euler-method} satisfies
\begin{gather}
 \mathbb{E} \sup_{t\in [0, T]} \left| X_t^{h} \right|^q  \leq  C(q,T)\left(1+\mathbb{E}\left|X_0\right|^q \right).
 \end{gather}
 Consequently, for those $p$ values in Assumption \ref{ass:Uniform bound}  \tcb{(i.e. $p \geq \max (p_1, p_2, p_3)$)}, 
 $$ \mathbb{E} \sup_{t\in [0, T]}
\left| X_t^{h} \right|^p\le C(p, T).$$
\end{lemma}

\section{Estimates of the logarithmic numerical density}\label{sec:nablalog}

In this section, we explore the $L^p$ upper bounds for derivatives of the logarithmic numerical density, leveraging tools from Malliavin calculus. By analyzing Malliavin matrices and derivatives, we derive integrated estimates for the logarithmic Green's function. These results are critical for establishing sharp error bounds in both additive and multiplicative noise settings.


For the discrete scheme \eqref{eq:Euler-method}, it is hard to use the technique of the differential equations, so we consider the following time continuous interpolation over the interval $t \in[ t_{k} , t_{k+1}]$ via
\begin{equation}\label{eq:continuous-Euler-method}
X^{h}_t:=X_{ t_{k} }^{h} 
+\int_{t_{k} }^{t} b\left(t_{k}, X_{ t_{k } }^{h} \right) d s +
\int_{ t_{k} }^t \sigma ( t_{k}, X_{ t_{k} }^{h} )
 d W_s .	
\end{equation}
Let us denote the densities (Radon-Nikodym derivative of the laws with respect to the Lebesgue measure) of $X_t$ for \eqref{eq:SDE_multiplicative} and $X^{h}_t$ for \eqref{eq:continuous-Euler-method} by $\rho_t$ and $\hat{\rho}_t$,  respectively. 
As we have mentioned in the introduction, using the strong error of the schemes, it is possible to obtain the half-order error bound
$O(\sqrt{h})$ for the densities, which is not sharp. Alternatively, in order to improve the error bound, in what follows, we turn our attention to the time evolution of the relative entropy $\cH\left( \hat{\rho}_t | \rho_t\right)$.

\subsection{Setup and the main estimates}

Consider the Green's function (transition probability) 
\begin{gather}
p(t,x; {a},y):=p(X_t^h=x | X_{{a}}^h=y)
\end{gather}
 of the time-continuous interpolation of the Euler method \eqref{eq:continuous-Euler-method} for $0\le {a}\le t$ where ${a}$ is a grid point ${a}=t_m$. This is in fact the law of the solution to the following equation 
\begin{equation}\label{eq:Eulercont_green}
	Y^{h}_t = y +\int_{{a}}^t b(\lfloor s \rfloor,  Y^{h}_{ \lfloor s \rfloor } ) d s
	+\int_{{a}}^t \sigma (\lfloor s \rfloor, Y^{h}_{\lfloor s \rfloor }) d W_s ,
\end{equation}
where we have denoted the following
\begin{gather}
	\lfloor s \rfloor := h [s/h],
\end{gather}
and $[x]$ indicates the largest integer no greater than $x$. Here, we will consider general two time points which could be in different time intervals. For convenience, we introduce the stochastic Jacobian and the stochastic Hessian
\begin{gather}
(J_{t}^h)_{ij}=\partial_{y_i}(Y_t^h)^j, \quad (H_t^h)_{ijk}=\frac{\partial^2}{\partial y_i \partial y_j}(Y_t^h)^k.
\end{gather}

With the Green's function, we find that
\begin{gather}\label{eq:densityformula}
\hat{\rho}_t(x)=\int_{\R^d} p(t, x ; 0, y) \rho_{0}(y)\,dy.
\end{gather}
Comparing to the heat equation, one may expect that $\nabla_x p(t, x ; 0, y)
\sim \frac{x-y}{t} p(t, x; 0, y)$. The singularity of $t$ near $t_0=0$ is large so direct estimation of $\nabla_x\hat{\rho}_t(x)$ is troublesome using only the integrability of $\rho_0$.

The intuition, as mentioned above, is that $\nabla_x\sim -\nabla_y$ for the Green's function. The derivative on $y$ then can be
moved onto the initial density. Hence, using the smoothness of the initial density, one can do the estimates.
This in fact corresponds to the averaging effect of the diffusion processes. 
With this understanding, we introduce the following:
\begin{gather}\label{eq:firstordercommn}
R(t,x; {a},y) := (\nabla_x + \nabla_y)\log p(t,x; {a},y),
\end{gather}
and
\begin{gather}\label{eq:secondordercommn}
R_2(t,x; {a},y) := (\nabla_x + \nabla_y)^2\log p(t,x; {a},y),
\end{gather}
\tcb{respectively.} Here, $R$ is a vector-valued function and $R_2$ is a matrix-valued function.
The intuition is that these two functions are well-behaved near $t=0$ using which we derive the main results in this section.

For the original continuous SDE in 1-dimension, we are able to obtain the pointwise estimate of $R$ and $R_2$
(see  Appendix \ref{app:pointwise}).
However, for numerical schemes, the pointwise estimate of $R$ and $R_2$ seems challenging
(see the final remark in Appendix \ref{app:pointwise}).
Instead, we prove the following integrated versions, which are already enough for many applications including our sharp convergence analysis in Section \ref{sec:mainresult} below. To simplify the notation, we denote the polynomial 
\begin{gather}\label{eq:Pk}
\mathscr{P}_k(y) := 1 + |y|^k
\end{gather} for $k > 0$ and $y \in \mathbb{R}^d$.

\begin{theorem}\label{thm:RandR2}
Suppose Assumption \ref{ass:condition-drift-diffusion} holds. Let $R$ and $R_2$ be defined as in  \eqref{eq:firstordercommn}
and \eqref{eq:secondordercommn}. Then, for any $p\ge 1$ and $T>0$, there exist  positive constants $C_1 = C(p,T)$ and $\nu = c(1+pq)$ (for some universal constant $ c $ and $ q $ in Assumption \ref{ass:condition-drift-diffusion}) such that for all $h$, $0\le {a} \le t\le T$ with ${a}=t_m$ being a grid point and any $y$, it holds that 
\begin{equation}
\mathbb{E}\left[ |R(t,Y^h_t; {a},y)|^p  \right]+\mathbb{E}\left[ |R_2(t,Y^h_t; {a},y)|^p  \right] \leq  C_1 \mathscr{P}_{\nu}(y).
\end{equation}
\end{theorem}

\begin{theorem}\label{thm:derive-density}
Suppose Assumption \ref{ass:condition-drift-diffusion} holds. Recall the numerical density $\hat{\rho}_t(x)$ for \eqref{eq:continuous-Euler-method}.  For $p\ge 1$ and $t>0$,  it holds that
\begin{multline}\label{eq:firstderivativecontrol}
\int_{\R^d} \big| \nabla \log \hat{ \rho }_{t}(x ) \big|^p \hat{\rho}_t(x)\,dx \leq
		C \left(\int_{\R^d} \big| \nabla \log \rho_0(y) \big|^p \rho_0(y)\,dy
				+ \mathbb{E}\left[ |R(t,X^h_t; 0,X_0^h)|^p\right]  \right),
\end{multline}
where $C$ depends only on $p$. The expectation is taken over the joint law for $(X_0^h, X_t^h)$. Similarly, it holds that
\begin{multline}\label{eq:secondderivativecontrol}
\int_{\R^d} \big| \nabla^2 \log \hat{ \rho }_{t}(x ) \big|^p \hat{\rho}_t(x)\,dx
		\leq C \big( \int_{\R^d} \big| \nabla^2 \log \rho_0(y) \big|^p \rho_0(y)\,dy+\mathbb{E}\big[ \big| R_2(t,X^h_t;0,X_0^h)\big|^{p} \big]\\
+\int_{\R^d}|\nabla\log\rho_0(y)|^{2p}\rho_0(y)\,dy+ \mathbb{E} \big[ \big| R(t,X^h_t;0,X_0^h) \big|^{2p} \big]  \big),
\end{multline}
where the expectation has the same meaning.
\end{theorem}
Consequently, for all $t\in [0, T]$, and for $p_2, p_3$ in Assumption 
\ref{ass:condition-drift-diffusion} and Assumption \ref{ass:Uniform bound},
one has  
\begin{gather}\label{eq:integratedlogdensity}
\mathbb{E} \big[ \left|\nabla\log\hat{\rho}_{t}(X_{t}^{h} ) \right|^{p_2} \big] +\mathbb{E} \big[\left|\nabla^2\log\hat{\rho}_{t}( X_{t}^{h} ) 
\right|^{p_3} \big] \le C.
\end{gather}

We will prove these results in the subsequent subsections. The estimates for $R$ and $R_2$ are the core of this paper, which will be established by the Malliavin calculus.

\subsection{Proof of $L^p$ upper bounds for the logarithmic numerical density }

We first establish Theorem \ref{thm:derive-density} by assuming the estimates for $R$ and $R_2$ obtained in Theorem \ref{thm:RandR2}.
\begin{proof}[Proof of Theorem \ref{thm:derive-density}]
Let us consider the gradient. By \eqref{eq:densityformula}, we have
\begin{gather}
\begin{split}
\nabla_x\log\hat{\rho}_t(x) &=\int_{\R^d} \nabla_xp(t,x; 0,y)\frac{\rho_0(y)}{\hat{\rho}_t(x)}\,dy\\
&=\int_{\R^d} \left( R(t,x; 0, y)p-\nabla_yp(t,x; 0,y)\right)\frac{\rho_0(y)}{\hat{\rho}_t(x)}\,dy\\
&=\int_{\R^d} \left( R(t,x; 0, y)+\nabla_y\log\rho_0(y)\right)\frac{p(t,x; 0,y)\rho_0(y)}{\hat{\rho}_t(x)}\,dy.
\end{split}
\end{gather}
Note that $p(t,x; 0,y)\rho_0(y)$ is the joint law for $(X_0^h, X_t^h)$ so $p(t,x; 0,y)\rho_0(y)/\hat{\rho}_t(x)$
is the conditional law of $X_0^h$ on $X_t^h=x$. Using Jensen's inequality, one has for $p\ge 1$ that
\[
|\nabla_x\log\hat{\rho}_t(x)|^p
\le \int_{\R^d} \left| R(t,x; 0, y)+\nabla_y\log\rho_0(y)\right|^p \frac{p(t,x; 0,y)\rho_0(y)}{\hat{\rho}_t(x)}\,dy.
\]
Then, using the polynomial (in $y$) upper bound for $R$, and Assumption \ref{ass:Uniform bound}, taking the integral with respect to $\hat{\rho}_t(x)$ gives the first result.

The second-order derivative is similar. First, one notes that
\[
R_2(t, x; 0, y) p(t, x; 0, y)=(\nabla_x+\nabla_y)^2p-p(t, x; 0, y) R(t, x; 0, y)\otimes R(t, x; 0, y).
\]
One then has
\begin{gather}
\begin{split}
\nabla^2\log\hat{\rho}_t(x)
&=\int_{\R^d} \nabla_x^2p(t,x; 0,y)\frac{\rho_0(y)}{\hat{\rho}_t(x)}\,dy-\nabla\log\hat{\rho}_t(x)^{\otimes 2}\\
&=\int_{\R^d}(R_2+R\otimes R)p(t,x; 0,y)\frac{\rho_0(y)}{\hat{\rho}_t(x)}\,dy+\int_{\R^d} p(t,x; 0,y)\frac{\nabla_y^2\rho_0(y)}{\hat{\rho}_t(x)}\,dy\\
&\quad +\int_{\R^d} \left( R\otimes\nabla_y\log\rho_0(y)+\nabla_y\log\rho_0(y)\otimes R\right)p(t,x; 0,y)\frac{\rho_0(y)}{\hat{\rho}_t(x)}\,dy\\
&\quad -\nabla\log\hat{\rho}_t(x)^{\otimes 2}.
\end{split}
\end{gather}
Using $\nabla_y^2\rho_0(y)=\rho_0(y)\nabla_y^2\log\rho_0(y)+\rho_0(y)(\nabla_y\log\rho_0(y))^{\otimes 2}$, taking the $p$-th moments of the above quantity with respect to $\hat{\rho}_t(x)$, one has by Jensen's inequality that \eqref{eq:secondderivativecontrol}
is controlled by the right hand side plus $\int |\nabla\log\hat{\rho}_t(x)|^{2p}\hat{\rho}_t(x)\,dx$, which can be controlled by \eqref{eq:firstderivativecontrol}. This then finishes the proof of the second claim.

At last, we note the facts such as $\E(|R(t, X_t^h, 0, X_0^h)|^p | X_0=y)=\E|R(t, Y_t^h, 0, y)|^p$, \eqref{eq:integratedlogdensity}  holds by a direct application of Theorem \ref{thm:RandR2}.
\end{proof}

To prove Theorem \ref{thm:RandR2}, we need some preparations. Recall the process $Y_t^h$ for the Green's function in \eqref{eq:Eulercont_green}. Since the initial value at ${a}$ is fixed to be $y$, $Y_t^h$ is determined by the Brownian motion totally, we then apply the Malliavin calculus to obtain the formulas for $R$ and $R_2$ (recall the notations for Malliavin calculus  in section \ref{subsec:malliavin}). Below, we do not write out the dependence on the initial time ${a}$ in $Y_t^h$ for convenience.

Here, since $Y_t^h$ is valued in $\R^d$, the Malliavin derivative will be a matrix. Let us fix the convention that
$D_rY_t\in \R^{m\times d}$, that is the $j$th column is the Malliavin derivative of $(Y_t^h)^j$.
For the identity in \eqref{eq:byparts}, it is better to require that $u$ is also of size $m\times d$ and then we denote
\[
\langle A, B\rangle_H:=\int_{a}^{\infty}A^T B\,dt,\quad \langle\langle A, B\rangle\rangle_H:=\int_{a}^{\infty}\tr(A^T B)\,dt.
\]
The divergence is then understood as column-wise so that $\delta(u)$ is valued in $\R^d$. 
With this understanding, let us consider the so-called covering field related to a process $Y_t^h$. 
\begin{definition}
A field $u\in L^2(\Omega, L^2([{a},\infty), \R^{m\times d}) )$ is called a covering field for $Y_t^h$ if
for any smooth $\varphi: \R^d\to \R$, it holds that
\begin{gather}
\nabla\varphi(Y_t^h)=\langle D\varphi(Y_t^h), u\rangle_H.
\end{gather}
\end{definition}
For a process $v\in L^2(\Omega, L^2(a, \infty; \R^{m\times d}))$, if 
\[
\langle DY_t^h, v\rangle_H=\int_{a}^{\infty} (DY_t^h)^Tv\,dt
\]
is nonsingular, we can then choose
\begin{gather}\label{eq:coverfieldwithDX}
u= v \langle DY_t^h, v\rangle_H^{-1}\in L^2(\Omega; L^2({a}, T; \R^{m\times d})).
\end{gather}
If we want to study the dynamics up to time $T$ only, we can then take $v'= v1_{t\le T}$.
For studying the numerical solutions. However, if $v=DY_t^h$, then $v_r=0$ for $r>t$ so there is no need to take this truncation.

For a covering field $u$, let $u_i=u(:, i)\in \R^m$ be the $i$-th column. Then,  we have the following observation.
\begin{lemma}\label{lmm:Malliavinorder}
For $F\in L^2(\Omega; \R)$ that is measurable with respect to $\hat{\mathscr{F}}$, the following holds.
\begin{equation}
	\mathbb{E} \Big[\delta\big( u_k \delta ( u_i F ) \big)| Y_{t}^{h}  = x \Big]
	=\mathbb{E} \Big[ \delta\big( u_i \delta ( u_k F ) \big)	| Y_{t}^{h}  = x \Big].
\end{equation}

\end{lemma}
\begin{proof}
By the definition of the covering field $u$, one has for any smooth test function $\varphi$ that
\begin{multline*}
\E\left[\partial_{i k}\varphi(Y_t^h) F \right]=\E\left[\langle D\partial_k\varphi(Y_t^h), u_i \rangle_H F \right]
=\E\left[\langle D\partial_k\varphi(Y_t^h), u_i F \rangle_H  \right]\\
=\E\left[\partial_k\varphi(Y_t^h) \delta( u_i F )  \right]
=\E\left[\tcb{\langle D\varphi(Y_t^h), u_k\rangle_H} \delta( u_i F )  \right]=\E\left[ \varphi(Y_t^h) \delta(u_k \delta(u_i F))\right].
\end{multline*}
The left-hand side is symmetric in $i$ and $k$, and so is the right hand. Hence,
\[
\E\left[ \varphi(Y_t^h) \delta(u_k \delta(u_i F))\right]=\E\left[ \varphi(Y_t^h) \delta(u_i \delta(u_k F))\right].
\]
This thus gives the desired identity.
\end{proof}

\begin{lemma}\label{lmm:greenexpression}
Consider $R$ defined in \eqref{eq:firstordercommn}, and $R_2$ defined in \eqref{eq:secondordercommn} with ${a}=t_m$ being a grid point.  Let $Y_t^h(y)$ be the solution to \eqref{eq:Eulercont_green} (with initial value $Y_{{a}}^h=y$). Then we have
	\begin{equation}
		R(t,x; {a},y)  
		= \mathbb{E} ( \delta ( u ( (J_{t}^{h})^T - I ) ) | Y^{h}_t = x ) .
	\end{equation}
Recall that $(J_{t}^h)_{ij}=\partial_{y_i}(Y_t^h)^j$ is the stochastic Jacobian.
The $R_2$ satisfies the following. 
\begin{multline}\label{eq:R2usingU}
(R_2(t, x, a, y))_{\ell k}
= \mathbb{E} \Big[ 
	   \delta\big( u_k \delta ( u_{\ell} ((J_t^{h})_{k k}-1)((J_t^{h})_{\ell \ell} - 1 )  ) \big)
	   +\sum_{i=1,i\neq l}^d \delta\big( u_k \delta (u_i ((J_t^{h})_{k k}-1)(J_t^{h})_{\ell i}  ) \big) \\
+\sum_{j=1,j\neq k}^d \delta\big( u_j
\delta ( u_{\ell} (J_t^{h})_{ k j}((J_t^{h})_{\ell \ell} -1) ) \big)+
\sum_{i\neq l,j\neq k}^d \delta\big( u_i \delta ( u_j (J_t^{h})_{ k j} (J_t^{h})_{\ell i} ) \big)	\\
  +\sum_{i=1}^d \delta( u_i  (\nabla_y^2)_{\ell k} (Y_t^h)^{i} )   | Y_t = x   \Big]-R_\ell(t,x, a, y) R_k(t, x, a, y).
\end{multline}
\end{lemma}	
\begin{proof}
By the definition of the covering field, one has 
\[
 \mathbb{E} ( \nabla \varphi (Y^{h}_t) ) = \E(\tcb{\langle D\varphi(Y_t^h), u\rangle_H}) =\mathbb{E} ( \varphi (Y^{h}_t) \delta ( u ) ).
\]
\tcb{Note that the second equality above is due to the integration by parts formula in Malliavin calculus. Also, the tower property tells that $\mathbb{E} ( \varphi (Y^{h}_t) \delta ( u ) ) = \mathbb{E} \big(\mathbb{E}( \varphi (Y^{h}_t) \delta ( u ) \mid Y^h_t)\big)$. Recalling that the initial condition $Y^h_a = y$, then by definition one has that} 
\begin{gather*}
\int_{\mathbb{R}^d} \nabla \varphi (x) p(t,x; {a},y) d x  
		=\int_{\mathbb{R}^d}\varphi (x) \mathbb{E} ( \delta ( u ) | Y^{h}_t =x ) p(t,x; {a},y) dx.
\end{gather*}
\tcb{For the left, one does integration by parts with respect to $x$ and then obtains (note that the above holds for any test function $\varphi$)}
\[
\nabla_{x} \log p(t,x; {a},y)=-\mathbb{E}\left(\delta(u) | Y^{h}_t=x\right).
\]
	
Moreover, direct computation leads to
\begin{multline*}
\nabla_{y} \mathbb{E} \varphi\left(Y^{h}_t\right)=\mathbb{E} 
	\Big[ \sum_j \nabla_{y} Y^{h,j}_t \partial_j \varphi \left(Y^{h}_t(y)\right)\Big]
	=\sum_j \E[\nabla_y Y_t^{h,j}\tcb{\langle D\varphi(Y_t^h), u_j\rangle_H}]\\
	=\sum_j \E[\langle D\varphi(X_t^h),  \nabla_y Y_t^{h,j} u_j\rangle_H]
	=\E[\varphi(Y_t^h) \delta(u (J_t^h)^T)].
\end{multline*}
Similar computation gives
\[
\nabla_{y} \log p(t,x; {a},y)=\mathbb{E}\left(\delta\left(u (J^{h}_t)^T \right) | Y^{h}_t = x\right).
\]
The result for $R$ then follows.

Similar computation gives for 
$ \ell , k \in \{ 1, 2, \cdots , d \} $ the following. 
\begin{align*}
\bigg( \frac{\nabla_x^{2} p(t,x;{a},y) }{ p(t,x;{a},y)}\bigg)_{\ell k} 
	& = \mathbb{E} \Big[ \delta\big( u_k \delta ( u_{\ell} ) \big) | Y_{t}^{h} = x \Big] \\
	\bigg( \frac{\nabla_y \nabla_x p(t,x;{a},y) }{ p(t,x;{a},y)}\bigg)_{\ell k} 
	&= - \sum_{i=1}^{d} \mathbb{E} \Big[ \delta\big( u_k \delta ( u_i (J_t^h)_{\ell i} ) \big)| Y_{t}^{h}  = x \Big] \\
	\bigg(\frac{\nabla_x \nabla_y p(t,x;{a},y) }{ p(t,x;{a},y) }\bigg)_{\ell k} 
	& = - \sum_{j=1}^{d}\mathbb{E} \Big[
	\delta\big( u_{\ell} \delta ( u_j 
 (J_t^{h})_{ k j} ) \big)| Y_{t}^{h} = x \Big] \\
	\bigg(\frac{\nabla_y^{2}  p(t,x;{a},y) }{ p(t,x;{a},y) }\bigg)_{\ell k} 
	& =  \mathbb{E} \Big[ \sum_{i,j=1}^{d}
	\delta\big( u_i \delta ( u_j  (J_t^{h})_{ \ell i}  (J_t^{h})_{ k j} ) \big)
	+\sum_{i=1}^{d} \delta( u_i (\nabla_{y}^2)_{\ell k} (Y_t^{h})^{ i} ) | Y_{t}^{h} = x \Big].
\end{align*}
Then, the formula for $R_2$ follows by using Lemma \ref{lmm:Malliavinorder}.
\end{proof}

With expressions of $R$ and $R_2$ in hand, we will then prove Theorem \ref{thm:RandR2}.
The intuition is that both $J-I$ and $\nabla_y J$ are half order, $u$ is like $t^{-1}$, and each one $\delta(\cdot)$ contributes $t^{1/2}$. Hence, the orders in these terms exactly balance, as desired since $\nabla_x\sim -\nabla_y$.
To prove them, we list some results we will prove in later subsections. Recall that we take $u$ as in \eqref{eq:coverfieldwithDX}. 
Consider the Malliavin matrix
\begin{gather}\label{eq:discreteMmatrix}
G_{t}=\langle DY_t^h, DY_t^h\rangle_H=\int_{{a}}^{t} (D_r Y_t^h)^T(D_rY_t^h) dr,
\end{gather}
which is a $d\times d$ random matrix. In Theorem \ref{thm:Malliavin-matrix}, we will show for any $p>1$ that
\begin{gather}
 \mathbb{E}|G_t^{-1}|^{p} \leq C(t-{a})^{-p}.
\end{gather}
Here, $|G_t|$ is the spectral norm and $C$ does not depend on $h$.
Moreover, we have the following estimates, and the sketch of the proof will be deferred to section \ref{subsec:proofother}.
\begin{proposition}\label{eq:otherMalliavin}
Fix $T>0$. Fix a positive integer $p\ge 2$. Suppose Assumption \ref{ass:condition-drift-diffusion} holds. Then for ${a}=t_m \in [0,T]$, and $t, s, u, r \in [{a}, T]$, there exist  
positive constants $C_1 = C(p,T)$ and $\nu := c(1 +pq)$ (for some universal constant $ c $ and $ q $ in Assumption \ref{ass:condition-drift-diffusion}) independent of ${a}$, $t$, $h$, $s$, $u$, $r$ such that (recall the definition of the polynomial $\mathscr{P}_k(y)$ in \eqref{eq:Pk})
\begin{enumerate}[(i)]
\item $\mathbb{E}|D_sY^h_{t}|^p \leq C_1\mathscr{P}_{p}(y),\quad \mathbb{E}|D_rD_sY^h_{t}|^p \leq C_1\mathscr{P}_{\nu}(y),\quad \mathbb{E}|D_uD_rD_sY^h_{t}|^p \leq C_1\mathscr{P}_{\nu}(y),$
\item $\mathbb{E}|J^h_t - I|^p \leq C_1(t-{a})^{\frac{p}{2}},\quad \mathbb{E}|D_sJ^h_t|^p \leq C_1\mathscr{P}_{\nu}(y),\quad \mathbb{E}|D_rD_sJ^h_t|^p \leq C_1\mathscr{P}_{\nu}(y),$
\item $\mathbb{E}|H^h_t|^p \leq C_1\mathscr{P}_{\nu}(y)(t-{a})^{\frac{p}{2}},\quad \mathbb{E}|D_sH^h_t|^p \leq C_1\mathscr{P}_{\nu}(y),$
\item $\mathbb{E}|D_sG_t|^p \leq C_1\mathscr{P}_{\nu}(y)(t-{a})^p,\quad \mathbb{E}|D_rD_sG_t|^p \leq C_1\mathscr{P}_{\nu}(y)(t-{a})^p,$
\item $\mathbb{E}|\delta(DY^h_{t})|^p \leq C_1\mathscr{P}_{\nu}(y)(t-{a})^{\frac{p}{2}},\quad \mathbb{E}|D_s\delta(DY^h_{t})|^p \leq C_1\mathscr{P}_{\nu}(y).$
\end{enumerate}
\end{proposition}
Now, we take the above claims for granted to show Theorem \ref{thm:RandR2}.

\begin{proof}[Proof of Theorem \ref{thm:RandR2}]

We first recall that $u=DY_t^h G_t^{-1}$ so that
\begin{gather}\label{eq:divergenceofcoveringfield}
\delta(u)=\delta(DY_t^h)\cdot G_t^{-1}-\langle\langle DY_t^h, DG_t^{-1}\rangle\rangle_H.
\end{gather}
It follows that for any $p>1$, one has by H\"older's inequality as well as by Proposition \ref{eq:otherMalliavin} that
\begin{gather}\label{eq:u1}
\E |u|^p\le (\E|DY_t^h|^{2p})^{1/2}(\E|G_t^{-1}|^{2p})^{1/2}\le C\mathscr{P}_{\nu}(y) (t-{a})^{-p},
\end{gather}
and
\begin{multline*}
\E|\delta(u)|^p\le C\mathscr{P}_{\nu}(y)\Big((\E|\delta(DY_t^h)|^{2p})^{1/2}(\E|G_t^{-1}|^{2p})^{1/2} \\
+(t-{a})^{p-1}\left(\int_{{a}}^t(\E|DY_t^h|^{2p})^{1/2}(\E|DG_t^{-1}|^{2p})^{1/2}\right) \Big).
\end{multline*}
The first term is controlled by $C\mathscr{P}_{\nu}(y)(t-{a})^{p/2}(t-{a})^{-p}\le C\mathscr{P}_{\nu}(y)(t-{a})^{-p/2}$.
Since $DG_t^{-1}=(G_t^{-1}D^{i}G_t G_t^{-1})_{i=1:m}$, one has
\begin{multline*}
\int_{{a}}^t(\E|DY_t^h|^{2p})^{1/2}(\E|DG_t^{-1}|^{2p})^{1/2}
\le C\mathscr{P}_{\nu}(y)\int_{{a}}^t (\E|G_t^{-1}|^{8p})^{1/4}(\E|DG_t|^{4p})^{1/4}\le C\mathscr{P}_{\nu}(y)(t-{a})^{-p+1}.
\end{multline*}
Hence, 
\begin{gather}\label{eq:u2}
\E|\delta(u)|^p\le C\mathscr{P}_{\nu}(y)(t-{a})^{-p/2}.
\end{gather}

Next, regarding $\mathbb{E}|R(t,Y^h_t; {a},y)|^p$,
by Lemma \ref{lmm:greenexpression}, Jensen's inequality and \eqref{eq:divergenceproduct}, one finds that
\begin{gather*}
\E |R(t,Y^h_t; {a},y)|^p\le \E|\delta(u((J_t^h)^T-I))|^p
\le C\left(\E|\delta(u)\cdot ((J_t^h)^T-I)|^p+\E|\langle\langle u, DJ_t^h\rangle\rangle_H|^p \right).
\end{gather*}
Applying H\"older's inequality and the results in Proposition \ref{eq:otherMalliavin}, together with \eqref{eq:u1}-\eqref{eq:u2}, the result then follows. As mentioned, the intuition is that $J-I$ is half order, which cancels the singularity near $t=0$ in $\delta(u)$ exactly, giving that
\[
\E |R(t,Y^h_t; {a},y)|^p\le C\mathscr{P}_{\nu}(y).
\]

For $R_2$, we first note that the 2-norm of a matrix can be controlled by its Frobenius norm, so it suffices to show that
\begin{equation*}
    \mathbb{E}|(R_2(t,Y^h_t;{a},y))_{\ell k}|^p \leq C\mathscr{P}_{\nu}(y),\quad \forall \ell,k.
\end{equation*}
The idea is quite straightforward and similar by using Proposition \ref{eq:otherMalliavin}, but the calculation is more tedious. 
We skip the details.
\end{proof}

\subsection{$L^p$ estimate for inverse of Malliavin matrix}\label{subsec:inverseMmatrix}

Here, we will also fix the initial time ${a}=t_m$ to be some grid point. 
Suppose that $t\in (t_{n-1}, t_n]$. We denote 
\begin{gather}
\delta t=t-\lfloor t^-\rfloor,\quad \lfloor t^-\rfloor=\lim_{s\to t^-}\lfloor s\rfloor.
\end{gather}
This definition makes $\delta t_n=h$ for a grid point $t_n$. We also denote 
\begin{gather}
\delta W_t:=W_t-W_{t-\delta t},
\end{gather}
and $\nabla b_n:=\nabla b(t_n, Y_{t_n}^h)$.
Recall that $D_r Y_t^h\in \R^{m\times d}$, then
\begin{gather}
D_r Y_t^h=
	\begin{cases}
	       0 & r>t.\\
		\sigma_{n-1}^T & r\in (t_{n-1}, t],\\
		D_rY_{t_{n-1}}^h+D_rY_{t_{n-1}}^h\cdot(\nabla b_{n-1}\delta t
		+\nabla\sigma_{n-1}\cdot \delta W_t) & r\le t_{n-1}.
	\end{cases}
\end{gather}

Then, the Malliavin matrix (as $\R^{d\times d}$ valued) is given by
\begin{multline}
	G_{t}=\int_{{a}}^{t} (D_r Y_t^h)^T(D_r Y_t^h) dr
	=\delta t \Lambda_{n-1}+\\
	(I+\nabla b_{n-1}\delta t +\nabla\sigma_{n-1}\cdot \delta W_t )^TG_{t_{n-1}} (I+\nabla b_{n-1}\delta t
	+\nabla\sigma_{n-1}\cdot \delta W_t ).
\end{multline}

Below we give the estimate of $\E | G_{t_{n}}^{-1}|^p$.

\begin{theorem}\label{thm:Malliavin-matrix}
Suppose Assumption \ref{ass:condition-drift-diffusion} holds. There exists $h_0>0$ such that for all $h\le h_0$ and any $p>1$, the Malliavin matrix $ G_{t} $ satisfies
	\begin{equation}
		\E | G_{t}^{-1} |^p \leq C (t-a)^{-p},
	\end{equation}
where $C$ is independent of $h$ and $a, t$ (may depend on $h_0, T, p$).
\end{theorem}

First of all, the spectral norm $|\cdot |$ does not have good properties for derivatives. We choose to use the Frobenius norm given by
\begin{gather}
|A|_F^2:=\sum_{i, j}|a_{ij}|^2,
\end{gather}
since it is easy to compute the derivatives of this norm and 
$|G_{t}^{-1}|\le |G_{t}^{-1}|_F$.
Note that we do not use the usual notation $\|\cdot\|_F$ for the Frobenius norm as we will reserve $\|\cdot\|$ for the norm over the probability space. A simple property we will use is that if $0\prec A \preceq B $
then
\begin{gather}
|A^{-1}|_F\ge |B^{-1}|_F.
\end{gather}
A difficulty in the computation below is that $(I+\epsilon C)^T \Phi (I+\epsilon C)\succeq (1-c\epsilon)\Phi$
does not hold for $c$ that is independent of the condition number of $\Phi$ if $C$ is simply a bounded matrix. The possible rotation could introduce the trouble.
To treat the possible rotation below, we first change $\Lambda$ to the identity matrix. To this end, we consider
\begin{gather}\label{eq:tildephi}
	\tilde{G}_t:=\delta t  I+F_{t}^T\tilde{G}_{t_{n-1}}F_{t},\quad \tilde{G}_{a}=0
\end{gather}
where
\begin{gather}
F_t:=I+\nabla b_{n-1}\delta t+\nabla\sigma_{n-1}\delta W_t.
\end{gather}
Then, it is not hard to see that 
\begin{gather}
|G_{t}^{-1}|_F\le \kappa^{-1} |\tilde{G}_t^{-1}|_F.
\end{gather}

We give the following key lemma. 
\begin{lemma}\label{lem:matrix-recursion}
Let $\lambda>0$ and $G$ be a positive definite matrix. For any $t\in ({a}, T]$, we set $k$ be the largest integer
such that $t_k<t$ (if $t=t_n$ for example, then $k=n-1$), then the following holds for some positive constants $C$
that is independent of $h$, $t$:
\begin{gather}
\E \left[ | (\lambda I+F_{t}^TG F_{t})^{-1}|_F^{p} | \mathcal{F}_{k}\right]
		\le e^{C\delta t} |(\lambda (1-C \delta t)^2I+G)^{-1}|_F^{p}+C\lambda^{-p}\delta t^{p+2}.
\end{gather}
\end{lemma}

\begin{proof}

Direct taking the derivative of $| (\lambda I+F_{t}^TG F_{t})^{-1}|_F^{p}$ in fact brings difficulty
due to the possible rotation in $F_t$. To resolve this, we will essentially use the identity
$F_t^{-1}(\lambda (F_tF_t^T)^{-1}+G )^{-1}F_t^{-T}$. This is the reason why we bound $\Lambda$ using $\kappa I$ above
since then $F_tF_t^T$ will be convenient to treat with.

To put the analysis rigorous, we use the continuous It\^o calculus. We need to consider the following stopping time
\begin{gather}
    \tau=\sup_{s<\delta t}\left\{\sup_{0<s' \le s} | W_{s'}|<\sqrt{\gamma_p \delta t |\log \delta t|}\right\}.
\end{gather}
Due to the reflection principle of the Brownian motion, it is clear that we can take $\gamma_p$ to be a constant independent 
of $\delta t$ such that
\begin{gather}
\P(\tau<\delta t)\lesssim \P(|W_{\delta t}|\ge \sqrt{\gamma_p \delta t |\log \delta t|}\})
\le (\delta t)^{p+2}.
\end{gather}

Let $W_s^{\tau}:=W_{s\wedge \tau}$, and  consider the continuous interpolation.
	\begin{gather}
		\begin{split}
			&\bar{f}(F(s))=|(\lambda I + F(s)^T GF(s))^{-1}|_F^p,\\
			& F(s)=I-\nabla b_{k} \delta t +\nabla\sigma_k W_{s}^{\tau}.
		\end{split}
	\end{gather}
Then, it suffices to estimate $\E [\bar{f}(F(\delta t))|\mathcal{F}_{k}]$ because
\[
\left|\E [\bar{f}(F(\delta t))|\mathcal{F}_{k}]-\E \left[ | (\lambda I+F_{t}^TG F_{t})^{-1}|_F^{p} | \mathcal{F}_{k}\right]
\right| \le C\lambda^{-p} (\delta t)^{p+2}.
\]
	
As we have mentioned, we pull $F(s)$ out and get
\[
(\lambda I + F(s)^T G F(s))^{-1} =F(s)^{-1}(\lambda (F(s)F(s)^T)^{-1}+G)^{-1} F(s)^{-T}.
\]
Since
\[
\tcb{|C A C^T|_F}=\sqrt{\mathrm{tr}(C A C^T CA C^T)}=\sqrt{\mathrm{tr}( A C^T CA C^TC)},
\]
we consider the following function
\begin{gather*}
	\bar{f}(F(s))=\left( \tr(A(s)(\lambda A(s)+G)^{-1}A(s)(\lambda A(s)+G)^{-1})\right)^{p/2}=:f(A(s), B(s)),
\end{gather*}
	where
	\begin{gather*}
		A(s):=(F(s)F(s)^T)^{-1}, \quad B(s):=(\lambda A(s)+G)^{-1}.
	\end{gather*}
	Next, we aim to apply It\^o's formula to the above function by noting that the functions are well-defined for the stopped process.

	It is clear that
	\begin{gather*}
		d(FF^T)=(\nabla\sigma_k dW_s^{\tau})F(s)^T+F(s)(\nabla\sigma_k dW_s^{\tau})^T
		+\sum_{j\ell}\nabla\sigma_{j\ell} \otimes (\nabla\sigma_k)_{j\ell} ds\wedge \tau.
	\end{gather*}
	Hence, one has
	\begin{gather*}
		dA(s)=-(FF^T)^{-1}\left((\nabla\sigma_k dW_s^{\tau})F(s)^T+F(s)(\nabla\sigma_k dW_s^{\tau})^T\right) (FF^T)^{-1}
		+\tilde{M}ds\wedge\tau,
	\end{gather*}
	where  $\tilde{M}$ is a bounded matrix since $F(s)$ is also uniformly
  bounded  (the concrete expression is not important to us).
	Then, one can compute
	\begin{gather*}
		dB(s)=-\lambda B(s)dA(s)B(s)+\lambda^2B(s) \tilde{N} B(s) ds\wedge \tau,
	\end{gather*}
	with
	\[
	|\tilde{N}|\le C|B(s)|,\quad   |\tilde{N}|_F \le C|B(s)|_F ,
	\]
	where $C$ is independent of $h$ and $\lambda$. Here, $|\cdot |$ indicates the 2-operator norm of the matrix.

For $f(A, B)=(\tr(ABAB))^{p/2}$ with $A, B$ being positive definite matrices, one has
	\begin{gather*}
		\frac{\partial f}{\partial a_{ij}}=p (\tr(ABAB))^{p/2-1}\tr(E_{ij}BAB)=p (\tr(ABAB))^{p/2-1}(BAB)_{ij},
	\end{gather*}
where $E_{ij}$ denotes the matrix that is $1$ only at the $(i, j)$-th entry.   Since $f(\cdot, \cdot)$ is a symmetric function, one has
	\begin{gather*}
		\frac{\partial f}{\partial b_{ij}}=p (\tr(ABAB))^{p/2-1}(ABA)_{ij}.
	\end{gather*}
The second-order derivatives are some fourth-order tensors, given by the following.
	\begin{gather*}
		\begin{split}
			\frac{\partial f^2}{\partial a_{\alpha\beta}\partial a_{ij}}
			&=p(p-2) (\tr(ABAB))^{p/2-2}(BAB)_{ij}(BAB)_{\alpha\beta}
			+p (\tr(ABAB))^{p/2-1}(BE_{\alpha\beta}B)_{ij},\\
			\frac{\partial f^2}{\partial b_{\alpha\beta}\partial b_{ij}}
			&=p(p-2) (\tr(ABAB))^{p/2-2}(ABA)_{ij}(ABA)_{\alpha\beta}
			+p (\tr(ABAB))^{p/2-1}(AE_{\alpha\beta}A)_{ij},\\
			\frac{\partial f^2}{\partial b_{\alpha\beta}\partial a_{ij}}
			&=p(p-2) (\tr(ABAB))^{p/2-2}(BAB)_{ij}(ABA)_{\alpha\beta}\\
			&\quad\quad+p (\tr(ABAB))^{p/2-1}(E_{\alpha\beta}AB)_{ij}
			+p (\tr(ABAB))^{p/2-1}(BAE_{\alpha\beta})_{ij}.
		\end{split}
	\end{gather*}
	
	By It\^o's formula, one has
	\begin{multline*}
		f(A(s), B(s))=f(A(0), B(0))+p \int_0^{s\wedge\tau} (\tr(ABAB))^{p/2-1} BAB:dA(s_1)\\
		+p\int_0^{s\wedge\tau} (\tr(ABAB))^{p/2-1} ABA: dB(s_1)
		+\frac{1}{2}\int_0^{s\wedge\tau}\Big(\frac{\partial f^2}{\partial a_{\alpha\beta}\partial a_{ij}} d[a_{ij}, a_{\alpha\beta}]_{s_1}\\
		+2\frac{\partial f^2}{\partial b_{\alpha\beta}\partial a_{ij}}d[a_{ij}, b_{\alpha\beta}]_{s_1}
		+\frac{\partial f^2}{\partial b_{\alpha\beta}\partial b_{ij}} d [b_{ij}, b_{\alpha\beta}]_{s_1}\Big),
	\end{multline*}
	where $[X, Y]_s$ denotes the quadratic covariant process for $X$ and $Y$.
	
	Taking expectation on both sides, the martingale terms involving It\^o's integral vanish so we focus on other terms.
	For the term involving $dA$, we have
	\begin{multline*}
		\E\int_0^{s\wedge\tau} (\tr(ABAB))^{p/2-1} BAB:dA(s_1)
		=\E\int_0^{s\wedge\tau} (\tr(ABAB))^{p/2-1} BAB:\tilde{M}ds\wedge\tau\\
		\le \E\int_0^{s\wedge\tau} (\tr(ABAB))^{p/2-1} \tr(BAB)ds\wedge\tau
		\le \E\int_0^{s\wedge\tau} (\tr(ABAB))^{p/2}\|A^{-1}\|ds\wedge\tau.
	\end{multline*}
	Here, we have used the inequality $\tr(PD)\le \tr(P)\|D\|$ if $P$ is a symmetric positive definite matrix.
	Since  $  \|A^{-1}\|$ is uniformly bounded, and $f$ is nonnegative, this term is controlled by 
	\[
	\E \int_{0}^{s\wedge\tau}f(A(s_1), B(s_1)) ds_1
	\le \int_{0}^s \E f(A(s_1), B(s_1)) ds_1.
	\]
	Similarly, the second term for $dB$ is controlled by
	\begin{multline*}
		\E\int_0^{s\wedge\tau} (\tr(ABAB))^{p/2-1} ABA:dB(s_1)\\
		=\E\int_0^{s\wedge\tau} (\tr(ABAB))^{p/2-1} ABA:B(-\lambda  \tilde{M}+\lambda^2\tilde{N})B ds_1.
	\end{multline*}
	Here, we note that $ABA:BDB=ABAB: DB$,
	and thus
	\begin{multline*}
		\E\int_{0}^{s\wedge\tau} (\tr(ABAB))^{p/2-1} ABA:B(-\lambda  \tilde{M}+\lambda^2\tilde{N})B ds_1\\
		\le \E\int_{0}^{s\wedge\tau} (\tr(ABAB))^{p/2} (\lambda+\lambda^2\|\tilde{N}\|)\|B\| ds_1.
	\end{multline*}
	Since  $\lambda \|B\|\le C$ which is independent of $\lambda$, $G$ and $h$, this term is also fine.

	For quadratic terms, the terms involving $dA$ are fine as we just use $BAB=A^{-1}(ABAB)$.
	There are several new terms. In $\frac{\partial f^2}{\partial b_{\alpha\beta}\partial b_{ij}}$, one needs to deal with
	\begin{gather*}
		\E \int_{0}^{s\wedge \tau}(\tr(ABAB))^{p/2-1}(AE_{\alpha\beta}A)_{ij} d[B, B]_{ij \alpha\beta}.
	\end{gather*}
	By the formula for $dB$ it is easy to see that 
	\[
	d[B, B]_{ij \alpha\beta}
	=\lambda^2\sum_{\theta} (BM_{\theta}B)_{ij}(BN_{\theta}B)_{\alpha\beta}\, ds\wedge\tau
	\]
	for some bounded matrices $M_{\theta}$ and $N_{\theta}$ where $\theta$ is for the labeling of different components of
	the Wiener process. Then, one actually has
	\[
	\sum_{\theta}\E \int_{0}^{s\wedge \tau}(\tr(ABAB))^{p/2-1}  \lambda^2 (ABN_{\theta}BA):BM_{\theta}B\,ds_1.
	\]
	Since $
	(ABN_{\theta}BA):BM_{\theta}B=BABN_{\theta}BAB:M_{\theta}$,
	we have
	\[
	(ABN_{\theta}BA):BM_{\theta}B\le
	\tr(BABN_{\theta}BAB)\|M_{\theta}\|
	\le \tr(BAB)\|M_{\theta}\|\|N_{\theta}\|\|B\|^2\|A\|.
	\]
	Moreover, since $BAB$ is symmetric positive definite
	\[
	\tr((BAB)^2)=\tcb{|BAB|_F}^2\le (\tr(BAB))^2\le \tr(ABAB)^2\|A^{-1}\|^2.
	\]
	 Since $\lambda^2\|B\|^2\le C$ independent of $\lambda$ and $G$ and $\|A^{-1}\|$ is bounded, this term is also fine.
	The term for $(BAE_{\alpha\beta})_{ij}$ is pretty much similar and we skip.
	
	Overall, we have established that
	\[
	\E f(A(s), B(s))\le f(A(0), B(0))+C \int_{0}^s\E f(A(s_1), B(s_1))\,ds_1.
	\]
	Applying Gr\"onwall's inequality, one has
	\[
	\E f(A(\delta t), B(\delta t))\le e^{C\delta t}f(A(0), B(0)) \le e^{C\delta t}\tcb{|A(0)|_F^p}.
	\]
	Note here that the constant $C$ has been enlarged (recall that $C$ is a generic constant). 
	Using the formula for $A(0)$, the desired claim follows since
	\[
	(I+\nabla b \delta t)(I+\nabla b \delta t)^T \succeq (1-C \delta t)^2I
	\]
	for some $C$ independent of $\delta t$.
\end{proof}

Now we apply Lemma \ref{lem:matrix-recursion} to prove Theorem \ref{thm:Malliavin-matrix}.
\begin{proof}
Recall that $a=t_k$ for some nonnegative integer $k$. Without loss of generality, we can assume $a=0$.
As we have mentioned, we only need to show for $p>1$ that
\[
\E |\tilde{G}_t^{-1}|_F^p\le Ct^{-p},
\] 
where $\tilde{G}$ is defined in \eqref{eq:tildephi}.

If $t\le h$, Lemma \ref{lem:matrix-recursion} is already the result. We consider $t>h$.
The result then follows by applying Lemma \ref{lem:matrix-recursion} recursively.
In fact, one first takes $\lambda_0=\delta t$ and then obtain ($n\ge 2$)
\[
\begin{split}
\E |\tilde{G}_t^{-1}|_F^p&\le e^{C\delta t}\E |(\delta t(1-C\delta t)^2I+\tilde{G}_{t_{n-1}})^{-1}|_F^{-p}
+C\delta t^2 \\
&=e^{C\delta t}\E |((\delta t(1-C\delta t)^2+h)I+F_{t_{n-1}}^T\tilde{G}_{t_{n-2}}F_{t_{n-1}})^{-1}|_F^{-p}
+C\delta t^2,
\end{split}
\]
where we have used the recursion relation in \eqref{eq:tildephi}.
Clearly, one can then set $\lambda_1=h+\delta t(1-C\delta t)^2$ and repeat the argument.
In general, one has $\lambda_k=(1-C h)^2\lambda_{k-1}+h$ for $k\ge 2$.
Then, one eventually has 
\[
\E |\tilde{G}_t^{-1}|_F^p\le e^{C(\delta t+(n-1)h)}
|\lambda_n I|_F^{-p}+C((n-1)h^2+\delta t^2).
\]
Then, it is easy to see that
\[
\lambda_n \ge h\sum_{i=0}^{n-1}(1-Ch)^{2i}\ge \frac{1-e^{-2C(n-1)h}}{2C-C^2h}.
\]
Since $t/[(n-1)h]\le 2$ and $ht\le C t^{-p}$, the result then follows.
\end{proof}

\subsection{Other auxiliary results}\label{subsec:proofother}

In this subsection, we close the argument by proving Proposition \ref{eq:otherMalliavin}.
That is, we need to give the desired bounds for the $p$-th moments of $DY^h$, $DDY^h$, $DDDY^h$, $J^h-I$, $DJ^h$, $DDJ^h$, $H^h$, $DH^h$, $DG$, $DDG$, $\delta(DY^h)$ and $D\delta(DY^h)$.

\begin{proof}[Proof of Proposition \ref{eq:otherMalliavin}]
The proof will follow the order of derivatives. In detail, we divide terms in Proposition \ref{eq:otherMalliavin} into the following classes:
\begin{itemize}
    \item Class 1 (first-order derivatives): $DY^h$, $J^h-I$,
    \item Class 2 (second-order derivatives): $DDY^h$, $DJ^h$, $H^h$, $DG$, $\delta(DY^h)$,
    \item Class 3 (third-order derivatives): $DDDY^h$, $DDJ^h$, $DH^h$, $DDG$, $D\delta(DY^h)$.
\end{itemize}

The reason for this classification is that estimates for terms in class $i$ rely on estimates for terms in Class $j$ ($j < i$). Without loss of generality, we can assume that $a=0$. If $t\le h$, these estimates are almost obvious (since it is one-step Euler).
Hence, we focus on $t>h$. Moreover, we can prove the results for $t_n=nh$. If the results have been proved for $t_n$, the general results for $t=t_n+\delta t$ with $\delta t\le h$ are obvious.
We will use the notations in section \ref{subsec:inverseMmatrix}. Moreover, the notation for the polynomial $\mathscr{P}_{\nu}(y)$ is generic and the order $\nu$ may vary from line to line. After all estimates are complete, we can take the maximum of these orders to redefine $\nu$.

\vskip 0.2 in

\noindent \textbf{STEP 1: Estimate terms in Class 1}

\noindent \textbf{Estimate for $DY^h$:}

By definition, one has
\begin{equation}\label{eq:DXdef}
D_r Y^h_{t} = \sigma(\lfloor r\rfloor , Y^h_{\lfloor r\rfloor})^T \textbf{1}_{\{r<t \}}
   + \sum_{j=0}^{n-1} D_r Y^h_{t_j} \cdot \left(h \nabla b_j + \nabla\sigma_j \cdot \delta W_{t_{j+1}}\right).
\end{equation}
Clearly, since $\nabla b$ is bounded, the estimate of the drift term is straightforward:
\begin{gather*}
    \mathbb{E}\left|\sum_{j=0}^{n-1} h D_r Y^h_{t_j} \cdot \nabla_{x} b_j\right|^p
    \leq Ch^p n^{p-1} \sum_{j=0}^{n-1}\mathbb{E}|D_r Y^h_{t_j}|^p \leq t^{p-1}h\sum_{j=0}^{n-1}\mathbb{E}|D_r Y^h_{t_j}|^p.
\end{gather*}
Note that the bound can in fact be improved to $(t-r)^{p-1}$, but we do not need this sharper bound.
For the Brownian motion term, we apply the Burkholder-Davis-Gundy (BDG) inequality (noting $\nabla \sigma_j(x)$ is bounded) to get
\[
\E|\sum_{j=0}^{n-1}D_r Y^h_{t_j}\cdot\nabla\sigma_j \cdot \delta W_{t_{j+1}}|^p\le C \E\left|\int_0^t|D_r Y^h_{\lfloor s \rfloor}|^2ds \right|^{p/2}
\le C\sum_{j=0}^{n-1}\mathbb{E}|D_rY^h_{t_j}|^ph,
\]
since $p\ge 2$. Hence, noting the Lipschitz assumption of $\sigma$, as well as the moment bound in Lemma \ref{lem:bound-moment}, we have
\begin{equation}\label{eq:gronwall1}
\mathbb{E}|D_r Y^h_{t_n}|^p \leq C\mathscr{P}_p (y) + Ch\sum_{j=0}^{n-1}\mathbb{E}|D_r Y^h_{t_j}|^p.
\end{equation}
The result then follows by the discrete Gr\"onwall's inequality.

\noindent \textbf{Estimate for $J^h-I$:}\\
By definition, $J_t^h$ satisfies
\begin{gather}\label{eq:Jdef1}
    J^h_t =I+ \sum_{j=0}^{n-1}J^h_{t_j}\cdot \Big(h\nabla b_j + \nabla \sigma_j\cdot \delta W_{t_{j+1}}\Big).
\end{gather}
so that 
\begin{gather}\label{eq:Jdef2}
    J^h_t-I = \sum_{j=0}^{n-1}(J^h_{t_j}-I  )\cdot \Big(h\nabla b_j + \nabla \sigma_j\cdot \delta W_{t_{j+1}}\Big)
    +\sum_{j=0}^{n-1}(h\nabla b_j + \nabla \sigma_j\cdot \delta W_{t_{j+1}}) .
\end{gather}
It is straightforward to see that (this is the Euler scheme of a simple SDE)
\[
\E|\sum_{j=0}^{n-1}(h\nabla b_j + \nabla \sigma_j\cdot \delta W_{t_{j+1}})|^p\le C
t^{p/2}.
\]
Using the same technique as we do for $D_rY_t^h$, the result then follows.

\vskip 0.2 in

\noindent \textbf{STEP 2: Estimate terms in Class 2}

\noindent \textbf{Estimate for $\delta(DY^h)$:}\\
By \eqref{eq:DXdef}, \eqref{eq:skorokhod} and the properties of the divergence, one has
\begin{equation}\label{eq:deltaDdef}
\begin{aligned}
\delta(DY^h_{t}) &= \int_0^t \sigma(\lfloor r\rfloor , Y^h_{\lfloor r\rfloor})^T dW_r+ \sum_{j=0}^{n-1} \delta \left(D Y^h_{t_j}\right)\cdot \left(h \nabla_{x} b_j + \nabla_{x} \sigma_j\cdot \delta W_{t_{j+1}}\right) \\
&\quad + \sum_{j=0}^{n-1} \tcb{\langle DY^h_{t_j}, D\left(h \nabla_{x} b_j + \nabla_{x} \sigma_j\cdot \delta W_{t_{j+1}}\right) \rangle_H}
\end{aligned}
\end{equation}
Using the polynomial growth conditions in Assumption \ref{ass:condition-drift-diffusion} and the moment bound in Lemma \ref{lem:bound-moment},
 the $L^p$ norm of the first term is clearly bounded. Note that by Lemma \ref{lem:bound-moment}, the bound is dependent of the initial value $ \mathscr{P}_{p}(y)$. Hence, it suffices to estimate the $L^p$ norm of the third term on the right-hand side of \eqref{eq:deltaDdef}. We note that
\begin{multline*}
 D_r\left(h\nabla_{x} b_j + \nabla_{x} \sigma_j\cdot \delta W_{t_{j+1}}\right)
    = hD_r Y^h_{t_j} \cdot \nabla^2_{xx}b_j 
    + D_r Y^h_{t_j}\cdot \nabla^2_{xx} \sigma_j \cdot \delta W_{t_{j+1}}+\nabla_{x}\sigma_j^T\textbf{1}_{[t_j,t_{j+1})}(r).
\end{multline*}
Clearly, 
\begin{equation*}
    \tcb{\langle DY_{t_j}^h, \nabla_{x} \sigma(t_j,Y^h_{t_j})^T\textbf{1}_{[t_j,t_{j+1})} \rangle_H} = 0,
\end{equation*}
because $D_rY_{t_j}^h=0$ for $r> t_j$. 
Then, using the polynomial growth conditions in \ref{ass:condition-drift-diffusion} and  the moment bound in Lemma \ref{lem:bound-moment}, it is straightforward to show that
\[
\E |\sum_{j=0}^{n-1} \tcb{\langle DY^h_{t_j}, D\left(h \nabla_{x} b_j + \nabla_{x} \sigma_j\cdot \delta W_{t_{j+1}}\right) \rangle_H}|^p
\le C\mathscr{P}_{\nu}(y) t^{p/2},
\]
where $\nu = O(1 + pq)$ and $q$ is the parameter in Assumption \ref{ass:condition-drift-diffusion}. The rest of the argument is the same as we have done for $DY_t^h$. 
We skip the details. Note that in the last step, we still need the discrete Gr\"onwall's inequality, and the coefficient in the forcing term is independent of $y$ (similar with \eqref{eq:gronwall1}), since we have assumed the global boundedness of $\nabla_x b$ and $\nabla_x \sigma$. Consequently, the final $L^p$ norm is still bounded by some polynomial $\mathscr{P}_{\nu}(y)$. This argument is also true for other estimates.

\noindent \textbf{Estimate for $DJ^h$ and $DDY^h$:}\\
By definition and \eqref{eq:Jdef1},
\begin{equation}\label{eq:DJdef}
D_r J^h_t = \sum_{j=0}^{n-1} D_rJ^h_{t_j}\cdot \Big(h\nabla b_j + \nabla \sigma_j \cdot \delta W_{t_{j+1}}\Big) + R,
\end{equation}
where
\[
R= \sum_{j=0}^{n-1}J_{t_j}^h\nabla\sigma_j 1_{[t_j, t_{j+1})}(r)
     + \sum_{j=0}^{n-1}\Big(hD_rY^h_{t_j}\cdot \nabla^2_{xx}b_j\cdot J^h_{t_j} )
    + J^h_{t_j}\cdot D_rY^h_{t_j} \cdot \nabla^2_{xx} \sigma_j \cdot \delta W_{t_{j+1}}   \Big).
\]
Noting the polynomial growth conditions in Assumption \ref{ass:condition-drift-diffusion}, the moment bound in Lemma \ref{lem:bound-moment} and the previous bounds, one can estimate $\E |R|^p$. Repeating the same argument, one has the desired estimate for $DJ^h$.
Using \eqref{eq:DXdef}, one finds that $D_rD_s Y^h_t$ has a similar structure  with $D_rJ_t^h$.
The argument is similar and the $L^p$ norm is uniformly bounded independent of $r$ and $s$.
We omit the details.

\noindent \textbf{Estimate for $DG$:}

By definition of $G$, one has
$D_rG_t=2\int_{{a}}^t \langle D_r^{\ell}D_uY^h_{t}, D_uY^h_{t}\rangle_{\ell=1:m} du$.
By the estimate of $DY^h$, $DDY^h$ (they have bounded $p$-th moments), one thus has by H\"older's inequality that
\begin{equation*}
    \mathbb{E}|D_rG_t|^p \leq C\mathscr{P}_{\nu}(y) t^p,
\end{equation*}
where $C$ is a positive constant independent of $t$, $h$, $r$.

\noindent \textbf{Estimate for $H^h$:}

Taking one more $y$-derivative in \eqref{eq:Jdef1}, one has
\begin{equation}\label{eq:Hdef}
H^h_t= \sum_{j=0}^{n-1} H^h_{t_j}\cdot \Big(h\nabla_{x} b_j+ \nabla_{x} \sigma_j \cdot \delta W_{t_{j+1}} \Big) + R_H,
\end{equation}
with
\begin{equation}
     R_H = \sum_{j=0}^{n-1} h J^h_{t_j}\otimes J^h_{t_j}:\nabla^2_{xx} b_j  + \sum_{j=0}^{n-1} J^h_{t_j}\otimes J^h_{t_j}:\nabla^2_{xx} \sigma_j\cdot \delta W_{t_{j+1}}.
\end{equation}
Here,
\[
J^h_{t_j}\otimes J^h_{t_j}:\nabla^2_{xx} :=\sum_{\ell, k} J_{\cdot \ell}\otimes J_{\cdot k}\partial_{x_{\ell}}\partial_{x_k}.
\]
As before, the $p$-th moment of $\sum_{j=0}^{n-1} (\cdots) \cdot  (W_{t_{j+1}} - W_{t_j})$ is bounded by $C\mathscr{P}_{\nu}(y)t^{\frac{p}{2}}$ by BDG inequality. Consequently, $\mathbb{E}|R_H|^p \leq C\mathscr{P}_{\nu}(y) t^{\frac{p}{2}}$.
Using the same argument as before, and noting the polynomial growth conditions in Assumption \ref{ass:condition-drift-diffusion} and the moment bound in Lemma \ref{lem:bound-moment}, one thus has
\begin{equation*}
\mathbb{E}|H^h_t|^p \le C\mathscr{P}_{\nu}(y) t^{\frac{p}{2}}.
\end{equation*}

\vskip 0.2 in

\noindent \textbf{STEP 3: Estimate terms in Class 3}

\noindent \textbf{Estimate for $DDJ^h$, $DH^h$, $DDDY^h$ and $DDG$}

These are pretty much similar to what we estimate for $DJ_t^h$. They of course would be much more tedious but the argument is similar, with the various remainder terms $R$ involving some terms as before, which have bounded $L^p$ norm.
We skip the details.

Using the definition of $G$ and the estimates we have for $DDY_t^h$ and $DDDY_t^h$, it is easy to see  $\E|D_sD_rG_t|^p \le C t^p$ where $C$ is independent of the time variables for the two Malliavin derivatives, similar as we have done for $DG$.

\noindent \textbf{Estimate for $D\delta(DY^h)$:}

Taking the Malliavin derivative in \eqref{eq:deltaDdef}, 
\begin{equation}\label{eq:ddeltaDdef}
D_r \delta(DY^h_{t})= \sum_{j=0}^{n-1} D_r \delta(DY^h_{t_j})\cdot  \Big( h\nabla_x b_j+\nabla_{x} \sigma_j \cdot \delta W_{t_{j+1}}\Big) + R_{DdD}.
\end{equation}
where
 \begin{multline*}
R_{DdD}= D_r\int_0^t \sigma(\lfloor r\rfloor , Y^h_{\lfloor r\rfloor})^T dW_r + 
\sum_{j=0}^{n-1}(\delta(DY_{t_j}^h)\cdot D_r^{\ell}(h\nabla b_j+\nabla\sigma_j \delta W_{t_{j+1}}))_{\ell=1:m}\\
+\sum_{j=0}^{n-1}D_r\langle DY_{t_j}^h, h DY_{t_j}^h\cdot\nabla^2b_j+DY_{t_j}^h\cdot\nabla^2\sigma_j \delta W_{t_{j+1}}\rangle_H.
 \end{multline*}
This time, $D_r$ hitting on the Brownian motions will contribute some $O(1)$ terms.
By expanding out the various derivatives, one can in fact find that
$\E|R_{DdD}|^p\le C\mathscr{P}_{\nu}(y)$.
Using the same argument as before and by the Gr\"onwall inequality, noting the polynomial growth conditions in Assumption \ref{ass:condition-drift-diffusion} and the moment bound in Lemma \ref{lem:bound-moment}, one has the desired estimate.
\end{proof}

\section{Sharp error estimate for Euler's scheme for SDEs with multiplicative noise}\label{sec:mainresult}

One crucial application of the estimates of the logarithmic numerical density derived in Section \ref{sec:nablalog} is to obtain a sharp error bound of Euler's scheme under the current settings. In this section, we establish the error bound in terms of the relative entropy while deferring the proofs for some technical estimates in later sections. The main idea is to use the time evolution equations for the numerical density.

\subsection{Statement of the sharp error order results}

Our another crucial finding in this paper is the following estimate for the relative entropy.
\begin{theorem}\label{thm:errentropy}
Let $T>0$ be the terminal time. Suppose Assumptions \ref{ass:condition-drift-diffusion} and 
\ref{ass:Uniform bound} hold. Consider the density $\hat{\rho}_t$
for the time continuous interpolation of the Euler method in \eqref{eq:Fokker-eq-numerical} and the density $\rho_t$ for the SDE in \eqref{eq:Fokker-eq-exact}. Then, there exist constants $h_0>0$ and $C(T)>0$ such that for all $h\le h_0$  the following holds:
\begin{equation}
\cH( \hat{\rho}_{t} |\rho_{t} ) \leq C h^2, \quad \forall t\in [0, T].
\end{equation}
\end{theorem}

As a consequence of the estimate for the relative entropy, we can obtain the bounds for the densities under some standard distances.
For example, the total variation distance can be bounded directly using the Pinsker's inequality as discussed in section \ref{subsec:relativeentropy}. 

\tcb{Regarding the Wasserstein distances, using the transport inequalities introduced in Section \ref{sec:prelim} above, one directly has the following claims.}
\begin{corollary}
\tcb{Consider the probability density functions 
$ \rho_{t} $, $ \hat{\rho}_{t} $ for 
$ X_t $, $ X^{h}_{t} $
defined in \eqref{eq:Fokker-eq-exact}
, \eqref{eq:Fokker-eq-numerical} with the same assumptions. If $\rho_t$ satisfies \eqref{eq:W1condition}, then}
\begin{equation}
	\tcb{W_1( \hat{\rho}_{t}, \rho_t) \leq C h.}
\end{equation}
\tcb{Also, if $\rho_t$ satisfies a log-Sobolev inequality with a uniform constant for $t\in [0, T]$ then}
\begin{equation}
	\tcb{W_2( \hat{\rho}_{t}, \rho_t) \leq C h.}
\end{equation}
\end{corollary}




\tcb{As a remark, for the $W_1$ distance, recall the weighted Csisz\'ar-Kullback-Pinsker inequality introduced near \eqref{eq:W1condition}. In order for \eqref{eq:W1condition} to hold, a sufficient condition is that the initial distribution $\rho_0$ is SubGaussian (i.e. there exists some $C>0$ such that $\mathbb{P}(|X_0| > a) \leq \exp(-a^2 / C^2)$ for all $a \geq 0$). In fact, under the current assumptions, it is easy to derive an equivalent characterization of SubGaussian property: $\mathbb{E}[\exp(\alpha |X_t|^2)] \leq 2$ for some positive $\alpha$. This then further means $\rho_t$ satisfies \eqref{eq:W1condition}.}

\tcb{For the $W_2$ distance, recall Talagrand transportation inequality introduced near \eqref{eq:talaW2}. In order for $\rho_t$ to satisfy a log-Sobolev inequality at some $t \geq 0$. A sufficient condition is to assume that the initial distribution $\rho_0$ satisfies a log-Sobolev inequality, and $-c_1<\frac{1}{2} \nabla^2:(\rho_0 \Lambda(t,\cdot)) -   \nabla \cdot( b(t,\cdot) \rho_0) <c_2$ for some $c_1$, $c_2 > 0$ and $t\in [0, T]$. In fact, denoting $q_t := \rho_t / \rho_0$. Note that under these conditions for $\rho_0$, the zero-th order term in the evolutionary equation for $q_t$ is bounded, so one has the maximal principle. Consequently, there exists $c>0$ such that $c^{-1} \leq q_t = \rho_t / \rho_0 \leq c$ for $t\in [0, T]$. By the classical Holley-Stroock perturbation lemma \cite[Proposition 5.1.6]{bakry2013analysis}, one knows that $\rho_t$ satisfies a log-Sobolev inequality.}

\subsection{Proof of the main results}

It is well-known that the SDE \eqref{eq:SDE_multiplicative} is well-defined under the assumptions, and the corresponding density $\rho_t$ satisfies the Fokker-Planck equation \cite{bogachev2022fokker} 
\begin{equation}\label{eq:Fokker-eq-exact}
\frac{\partial \rho_t}{\partial t} =-\nabla \cdot \left(\rho_t b\right)+\frac{1}{2}  \nabla^{2}: ( \rho_t \Lambda ).
\end{equation}
For the density $\hat{\rho}_t$, we have the following.
\begin{lemma}\label{lem:Fokker-eq}
	The density $\hat{\rho}_t$ of the process $X^{h}_t$ defined in equation \eqref{eq:continuous-Euler-method} satisfies the following equation
	\begin{equation}\label{eq:Fokker-eq-numerical}
		\frac{\partial \hat{\rho}_t}{\partial t}
		=
		-\nabla \cdot
		\left(\hat{\rho}_t \hat{b}_t\right)
		+\frac{1}{2} \nabla^{2} :
		(\hat{\rho}_t \hat{\Lambda}_{t} ),  
		\quad 
		t \in[t_{k},t_{k+1}],
	\end{equation}
	where 
 \begin{gather}\label{eq:backwardcond}
 \hat{b}_t(x):=\mathbb{E}\left(b\left( t_{k}, X_{ t_{k}}^{h}\right) | X^{h}_t=x\right), \quad \hat{\Lambda}_t(x):=\mathbb{E}\left( 
	\Lambda \left(t_{k}, X_{ t_{k}}^{h}\right) | X^{h}_t=x\right).
 \end{gather}

\end{lemma}

The proof of this lemma is similar to that in \cite[Lemma 1]{mou2022improved}, so we skip the details.
We remark that the backward conditional expectation in \eqref{eq:backwardcond} in general is not well-defined if there is no information of the density at previous time points. Here, the 
backward conditional expectation is OK since we know $\rho_0$ as the initial law for the numerical scheme.

We now present the proof for the main result.
\begin{proof}[Proof of Theorem \ref{thm:errentropy}]

\tcb{The proof of Theorem \ref{thm:errentropy} borrows ideas from \cite{li2022sharp} and \cite{mou2022improved}.}

{\bf Step 1. Time derivative of the relative entropy}

\tcb{We start from the Fokker-Planck type equations describing the evolution of the densities $\rho_t$ and $\hat{\rho}_t$. In fact,}
using \eqref{eq:Fokker-eq-numerical} and \eqref{eq:Fokker-eq-exact}, one can compute directly
\[
\begin{aligned}
&\frac{d}{d t} \cH(\hat{\rho}_t | \rho_t)  =\int_{\mathbb{R}^d} \frac{\partial \hat{\rho}_t}{\partial t}(\log \hat{\rho}_t+1-\log \rho_t) d x-\int_{\mathbb{R}^d} \frac{\partial \rho_t}{\partial t} \frac{\hat{\rho}_t}{\rho_t} d x \\
	& =  \tcb{-\int_{\mathbb{R}^d}\left(-\hat{\rho}_t \hat{b}_{t}+\frac{1}{2}\nabla \cdot (\hat{\rho}_t\hat{\Lambda}_{t}) \right) \cdot
	\nabla\log \frac{ \hat{\rho}_t}{\rho_t} dx +\int_{\mathbb{R}^d} \left(- (\rho_t b) + \frac{1}{2}  \nabla \cdot
		( \rho_t \Lambda )\right)  \cdot \nabla\left(\log \frac{ \hat{\rho}_t}{\rho_t}\right) \frac{ \hat{\rho}_t }{ \rho_t } d x} \\
	&  = \int_{\mathbb{R}^d}\hat{\rho_t } 
	( \hat{b}_{t} - b ) \cdot \nabla\log \frac{ \hat{\rho}_t}{\rho_t} d x 
	- \frac{1}{2} \int_{\mathbb{R}^d}( \hat{\Lambda}_{t} - \Lambda  ):\nabla \hat{\rho_t }\otimes \nabla\log \frac{ \hat{\rho}_t}{\rho_t}  d x \\
&  \quad - \frac{1}{2} \int_{\mathbb{R}^d} \hat{\rho}_t \left( \nabla \cdot \hat{\Lambda}_{t} 
	- \nabla \cdot   \Lambda \right)\cdot \nabla \log \frac{ \hat{\rho}_t}{\rho_t} d x 
- \frac{1}{2}\int_{\mathbb{R}^d}\hat{\rho}_t\Lambda:\left(\nabla\log \frac{ \hat{\rho}_t}{\rho_t}\right)^{\otimes 2} d x.
\end{aligned}
\]
The integration by parts can be justified by \eqref{eq:integratedlogdensity}.
Clearly, the last term gives the beneficial dissipation.
\begin{equation*}
- \frac{1}{2}\int_{\mathbb{R}^d}\hat{\rho}_t\Lambda:\left(\nabla\log \frac{ \hat{\rho}_t}{\rho_t}\right)^{\otimes 2} d x
\le  -\frac{1}{2}\kappa \int_{\mathbb{R}^d} \hat{\rho}_t \Big| \nabla \log \frac{\hat{\rho}_t}{\rho_t} \Big|^2 d x.
\end{equation*}
Using Young's inequality, we have
\begin{equation*}
\begin{split}
\frac{d}{d t} \cH(\hat{\rho}_t | \rho_t)
& \le C(\kappa)\Big[\int_{\mathbb{R}^d}\hat{\rho}_t | \hat{b}_{t} - b |^{2} d x +\int_{\mathbb{R}^d} \hat{\rho}_t 
			| ( \hat{\Lambda}_{t}- \Lambda)\cdot \nabla \log \hat{\rho}_t |^2  d x \\
	&\quad +\int_{\mathbb{R}^d}\hat{\rho}_t
			|\nabla \cdot \hat{\Lambda}_{t}
			-\nabla \cdot   \Lambda 
			|^{2}  d x\Big]
			-\frac{\kappa}{4}\int_{\mathbb{R}^d} \hat{\rho}_t \Big| \nabla \log \frac{\hat{\rho}_t}{\rho_t} \Big|^2 d x\\
   & =: J_1 +  J_2+  J_3 +  J_4,
\end{split}
\end{equation*}
for some constant $C(\kappa)>0$.

\tcb{Note that $J_4 \leq 0$. The other three terms are basically errors from the drifts and diffusions (note that $J_2$ and $J_3$ vanish if the noise is only additive). Similarly to the derivations in \cite{li2022sharp} and \cite{mou2022improved}, using integration by parts and Bayes' law, we can obtain an $O(h^2 (1 + (\mathbb{E}|\nabla \log \hat{\rho}_t(X^h_t)|^p)^{\tfrac{2}{p}}))$ upper bound for both $J_1$ and $J_2$ for some large $p$ (see more details in Step 2 below). Fortunately, in \eqref{eq:integratedlogdensity} above, we have already obtained an $O(1)$ upper bound for $\mathbb{E}|\nabla \log \hat{\rho}_t(X^h_t)|^p$, so $J_1$ and $J_2$ are of order $h^2$. The estimate for $J_3$ shares the similar idea but the detailed derivation is much more complicated. Moreover, it requires higher order regularity due to the existence of the divergence operator in it. In detail, by \eqref{eq:integratedlogdensity} again, we make use of the boundedness of the $L^p$ norm of $\nabla^k \log \hat{\rho}_t(X^h_t)$ for $k = 1,2$ and large $p$, and finally obtain an $O(h^2)$ upper bound for $J_3$.}

\vskip 0.2 in

\noindent {\bf Step 2. The $J_1$ and $J_2$ terms}

Consider $J_1$. From its definition, one has 
\begin{gather}\label{eq:diffbaux1}
 |\hat{b}_t(x)-b(t, x)|^2 
= \big|\mathbb{E}\big[b (t_{k}, X_{ t_{k}}^{h} )-b(t, X^{h}_t) | X^{h}_t=x\big] \big|^2.
\end{gather}
We decompose the difference
\[
b (t_{k}, X_{ t_{k}}^{h} )-b(t, X^{h}_t)=(b (t_{k}, X_{ t_{k}}^{h} )-b (t_{k}, X^{h}_t ))+(b (t_{k}, X^{h}_t )-b(t, X^{h}_t)).
\]
By Assumption \ref{ass:condition-drift-diffusion} and the bounded moment in Lemma \ref{lem:bound-moment},  one has
\begin{equation*}
    \int \big|\mathbb{E}\big[b (t_{k}, X^{h}_t )-b(t, X^{h}_t) | X^{h}_t=x\big] \big|^2\hat{\rho}_t dx 
    \leq \mathbb{E}\big[L_t^2(t-t_{k})^2(1 + |X^{h}_t|^q)\big] \leq C(t-t_{k})^2.
\end{equation*}

For the other term, by Taylor's expansion, one has
\begin{equation}\label{eq:J1axu1}
    \mathbb{E}\big[b (t_{k}, X_{ t_{k}}^{h} )-b(t_{k}, X^{h}_t) | X^{h}_t=x\big] 
    =  \mathbb{E}\big(X_{ t_{k}}^{h}-X^{h}_t | X^{h}_t=x \big)\cdot \nabla_x b(t_{k}, x)+\hat{r}_t(x),
\end{equation}
where the remainder takes the form
$$
\hat{r}_t(x):=\mathbb{E}\left(\int_0^1 (1-s)  (X_{ t_{k}}^{h}-X^{h}_t)^{\otimes 2}:\nabla_{x}^2b\left(t_{k}, (1-s) X^{h}_t+s X_{ t_{k}}^{h}\right) d s | X^{h}_t=x\right).
$$
By Jensen's inequality, the H\"older inequality and the moment control, the remainder term $\hat{r}_t$ satisfies
\begin{multline}\label{eq:J1auxR}
\int_{\R^d} |\hat{r}_t(x)|^2\hat{\rho}_t(x)\,dx
 \leq \E \left| \int_0^1 (1-s)  (X_{ t_{k}}^{h}-X^{h}_t)^{\otimes 2}:\nabla_{x}^2b\left(t_{k}, (1-s) X^{h}_t+s X_{ t_{k}}^{h}\right) d s\right|^2\\
 \le \sup_{0\le s\le 1} \Big(\mathbb{E}\big |  \nabla_{x}^2b\big( t_{k}, (1-s) X^{h}_t+s X_{ t_{k}}^{h} \big) \big|^{4}\Big)^{\frac{1}{2}}
 \Big(\E|X_{ t_{k}}^{h}-X^{h}_t |^{8}\Big)^{\frac{1}{2}}  \le C
 (t-t_{k})^2,
\end{multline}
where we have used the polynomial growth assumption for $\nabla^2_x b$ and the moment bound in Lemma \ref{lem:bound-moment} in the last inequality. 
For the first term on the right-hand side of \eqref{eq:J1axu1}, one has by Bayes' formula that
\begin{equation}\label{eq:error-one}
\mathbb{E}( X_{ t_{k}}^{h} - X^{h}_{t} | X^{h}_{t} = x 	)
	=\int_{\R^d} ( y - x )\frac{ \hat{\rho}_{t_{k} } (y) p(  X^{h}_{t} = x | X_{ t_{k}}^{h} = y) }{ \hat{\rho}_t (x)}  d y.
\end{equation}
Note that the transition kernel $p\left(X^{h}_t=x | X_{ t_{k}}^{h}=y\right) $ is a Gaussian given by
\begin{gather}
\begin{split}
& p\left(X^{h}_t=x | X_{ t_{k}}^{h}=y\right)= ( 2 \pi ( t - t_{k} ) )^{-\frac{d}{2}}\det (\Lambda(t_{k},y) )^{-\frac{1}{2}}\exp(-\phi(x, y, t))\\ 	
& \phi(x, y, t)=\frac{1}{2(t-t_{k})}( x - y - ( t - t_{k} ) b(t_{k},y) )^T\Lambda(t_{k},y)^{-1}( x - y - ( t - t_{k} ) b(t_{k},y) ).
\end{split}
\end{gather}

This elementary fact tells us that $(y-x)$ is roughly like the derivative of $y$ on $p(X^{h}_t=x|X^{h}_{t_{k}}=y)$. Integration by parts could treat this term, which corresponds to averaging out the fluctuation brought by the Brownian motion.  
In particular, we decompose 
\begin{gather}\label{eq:J1keyaux}
	(y-x)\exp(-\phi(x, y, t)) = - ( t - t_{k} ) \Lambda ( t_{k}, y) \nabla_{y} \exp(-\phi(x, y, t))+R_{1} \exp(-\phi(x, y, t)),
\end{gather}
where
\begin{equation}
	\begin{split}
		R_{1} & = -(t-t_k)b(t_k, y)-(t-t_k)\Lambda\cdot \nabla_y b\cdot \Lambda^{-1}(y-x+(t-t_k)b)\\
		 &- \frac{1}{2}\sum_{\ell=1}^d\Lambda_{\cdot, \ell}\big( x - y  - ( t - t_{k} ) b( t_{k}, y) \big)^{T}
		 \frac{\partial \Lambda^{-1} ( t_{k}, y) }{\partial{y_{\ell}} }\big( x - y  - ( t - t_{k} ) b( t_{k}, y)\big).
	\end{split}
\end{equation}

Hence, one has
\begin{gather}\label{eq:J1firstaux2}
	\begin{split}
		&\mathbb{E}( X_{ t_{k}}^{h} - X^{h}_{t} | X^{h}_{t} = x )\\
		&=\int_{\R^d}\frac{\hat{\rho}_{t_k}(y)}{\hat{\rho}_t(y)}\left(-(t-t_k)\Lambda\cdot \nabla_yp
		+[(t-t_k)\Lambda\cdot\nabla_y \log(\det (\Lambda(t_{k},y) )^{-\frac{1}{2}})+R_1] p\right)dy\\
		&=(t-t_{k})\mathbb{E}\left(
		\Lambda(t_k, X_{ t_{k}}^{h})\cdot\nabla_{y} \log \hat{\rho}_{t_{k}}\left(X_{ t_{k}}^{h}\right) | X^{h}_t=x\right)+\int R_2\frac{ \hat{\rho}_{t_{k} } (y)p(  X^{h}_{t} = x | X_{ t_{k}}^{h} = y) } { \hat{\rho}_t (x)}  d y.
	\end{split}
\end{gather}
where
$R_2=(t-t_{k})\nabla_{y} \cdot \Lambda (t_{k} , y) +(t-t_k)\Lambda\cdot\nabla_y \log(\det (\Lambda(t_{k},y) )^{-\frac{1}{2}})+R_1$.
Applying the Jensen's inequality , H\"older's inequality, the polynomial growth assumption and the moment bound in Lemma \ref{lem:bound-moment}, the $L^2$ norm of the first term on the right-hand side of \eqref{eq:J1axu1} is thus controlled by 
\[
\begin{split}
C (t-t_{k})^2 (\mathbb{E}\left|\nabla_{y}\log \hat{\rho}_{t_{k}}(X_{ t_{k}}^{h})\right|^{p_2})^{\frac{2}{p_2}} +2\E|R_2|^2 \le C (t-t_{k})^2,
\end{split}
\]
\tcb{where we have used 
$$ \mathbb{E}  \left|\nabla\log\hat{\rho}_{t}(X_{t}^{h} ) \right|^{p_2}   \le C  $$
in \eqref{eq:integratedlogdensity} above for $t = t_k$ (recall the $p_2$ in Assumption \ref{ass:Uniform bound})}, while the estimate for $R_2$ is relatively straightforward and we have skipped the details.  Hence, one has
\begin{equation*}
J_1	\leq C (t-t_k)^2.
\end{equation*}

By H\"older inequality, one has
\begin{equation}
J_{2} \le \left( \mathbb{E}  |  \hat{\Lambda}_{t}(X_t^h) - \Lambda(t, X_t^h) |^4 \right)^{\frac{1}{2}} \,\left( \mathbb{E} | \nabla \log \hat{\rho}_{t}(X_t^h) |^4\right)^{\frac{1}{2}}.
\end{equation}
\tcb{By \eqref{eq:integratedlogdensity} again, one knows that $\mathbb{E} | \nabla \log \hat{\rho}_{t}(X_t^h) |^4$ is bounded (recall that $p_2 > 4$), so} one only has to treat with $\hat{\Lambda}_{t}(X_t^h) - \Lambda(t, X_t^h)$. This can be done similarly as we did for $\hat{b}_t-b$,
except that we need the fourth moment. Hence, one also has
\begin{equation*}
J_2	\leq C(X_0) (t-t_k)^2.
\end{equation*}

\vskip 0.2 in

\noindent {\bf Step 3. The $J_3$ term.}

Consider $ J_{3} =\int_{\mathbb{R}^d}\hat{\rho}_{t}(x) |\nabla\cdot\hat{\Lambda}_{t}(x)
-\nabla \cdot \Lambda(t, x) |^{2}  d x$. We first have by definition that
\begin{gather*}
\begin{split}
\nabla\cdot\hat{\Lambda}_{t}(x)-\nabla \cdot \Lambda(t, x)&=\nabla_x\cdot \mathbb{E}\big[\Lambda (t_{k}, X_{ t_{k}}^{h} )-\Lambda(t, X^{h}_t) | X^{h}_t=x\big]\\
&=\nabla_x\cdot \mathbb{E}\big[\Lambda (t_{k}, X_{ t_{k}}^{h} )-\Lambda(t_k, X^{h}_t) | X^{h}_t=x\big]+
\nabla_x\cdot \big[\Lambda (t_{k}, x)-\Lambda(t, x) \big].
\end{split}
\end{gather*}

\tcb{By \eqref{eq:Lam_timelip} in Assumption \ref{ass:condition-drift-diffusion}, } the second term $\nabla_x\cdot \big[\Lambda (t_{k}, x)-\Lambda(t, x) \big]$ is easy to treat with and we skip.
Similarly as before, we have
\begin{equation}\label{eq:J4aux1}
    \mathbb{E}\big[\Lambda (t_{k}, X_{ t_{k}}^{h} )-\Lambda(t_{k}, X^{h}_t) | X^{h}_t=x\big] 
    =  \mathbb{E}\big(X_{ t_{k}}^{h}-X^{h}_t | X^{h}_t=x \big)\cdot \nabla_x \Lambda(t_{k}, x)+\bar{r}_t(x),
\end{equation}
where the remainder takes the form
\begin{gather*}
\begin{split}
\bar{r}_t(x)&=\mathbb{E}\left(\int_0^1 (1-s)  (X_{ t_{k}}^{h}-X^{h}_t)^{\otimes 2}:\nabla_{x}^2\Lambda\left(t_{k}, (1-s) X^{h}_t+s X_{ t_{k}}^{h}\right) d s | X^{h}_t=x\right)\\
&=\int_{\R^d}\int_0^1 (1-s)  (y-x)^{\otimes 2}:\nabla_{x}^2\Lambda\left(t_{k}, (1-s) x+s y\right) 
\frac{\hat{\rho}_{t_k}(y)p(X_t^h=x|X_{t_k}^h=y)}{\hat{\rho}_t(x)}\, dsdy.
\end{split}
\end{gather*}
We denote $z(s)=sx+(1-s)y$ and thus have
\begin{equation}\label{eq:divLambdaremainder}
\begin{split}
& \nabla_{x} \cdot \bar{r}_t(x) \\
= &\int_{\mathbb{R}^d}\int_0^1 (1-s)^2 (y -x)^{\otimes 2}: \nabla^2 (\nabla\cdot\Lambda) (z(s) )   p ( X _{ t }^{h} = x |
				X _{ t_{k} }^{h} = y )\frac{ \hat{\rho}_{t_k} ( y ) } { \hat{\rho}_{t} ( x ) } d s d y \\
				&- \int_{\mathbb{R}^d}\int_0^1 2 (1-s) (y-x)\cdot \nabla(\nabla\cdot
				\Lambda (z(s)) )p ( X _{ t }^{h} = x |X _{ t_{k} }^{h} = y )\frac{ \hat{ \rho }_{t_{k}}(y)}
				{ \hat{\rho}_{t} ( x ) } d s d y	\\
				&- \int_{\mathbb{R}^d}\int_0^1 (1-s) (y -x)^{\otimes 2}:\nabla^2 
				\Lambda (z(s) )\cdot \frac{ \nabla \log \hat{\rho}_{t} ( x ) }{ \hat{\rho}_{t} ( x ) }   p ( X _{ t }^{h} = x |X _{ t_{k} }^{h} = y )\hat{ \rho }_{t_k}(y) d s d y	\\
				& + \int_{\mathbb{R}^d}\int_0^1 (1-s) (y -x)^{\otimes 2}:\nabla^2 \Lambda (z(s) ) 
				 \cdot \nabla_x\log p ( \cdot|\cdot ) p ( X _{ t }^{h} = x |X _{ t_{k} }^{h} = y )
				\frac{\hat{ \rho }_{t_k}(y) }{ \hat{\rho}_{t} ( x ) } d sd y \\
=: & R_{31}+R_{32}+R_{33}+R_{34}.
	\end{split}
\end{equation}

\tcb{Note that at time $t$, using  
$$ \mathbb{E}  \left|\nabla\log\hat{\rho}_{t}(X_{t}^{h} ) \right|^{p_2}   \le C  $$
obtained in \eqref{eq:integratedlogdensity} (recall the $p_2$ in Assumption \ref{ass:Uniform bound}),
$\int \hat{\rho}_t(x)|R_{31}+R_{33}|^2$ can be estimated by Jensen's inequality similar as in \eqref{eq:J1auxR}.}
The $R_{32}$ term is similar to that in \eqref{eq:J1keyaux} and \eqref{eq:J1firstaux2}.  For $R_{3,4}$, we find that
\[
\nabla_x\log p ( \cdot|\cdot )=-\nabla_x\phi(x, y, t)=-\frac{1}{t-t_k}\Lambda(t_k, y)^{-1}(x-y-(t-t_k)b(t_k, y)).
\]
The part for $-\frac{1}{t-t_k}(t-t_k)b$ is easy which can be done similarly as $R_{31}$ as in \eqref{eq:J1auxR}.
The main  term in $R_{3,4}$,  $-\frac{1}{t-t_k}\Lambda(t_k, y)^{-1}(x-y)$, can be treated similarly as in  \eqref{eq:J1keyaux} and \eqref{eq:J1firstaux2}. That is,
\begin{multline*}
-\frac{1}{t-t_k}\Lambda(t_k, y)^{-1}(x-y)p ( X _{ t }^{h} = x |X _{ t_{k} }^{h} = y )
=-\nabla_yp ( X _{ t }^{h} = x |X _{ t_{k} }^{h} = y )\\
+\tilde{R}_{3,4}p ( X _{ t }^{h} = x |X _{ t_{k} }^{h} = y ),
\end{multline*}
where $\tilde{R}_{3,4}$ is easy to control.
This is not surprising as $\nabla_x\sim -\nabla_y$ to the leading order for the heat kernel, as discussed previously.
The $-\nabla_yp ( X _{ t }^{h} = x |X _{ t_{k} }^{h} = y )$ term, after integration by parts, gives the main terms
\[
R_{32}- \int_{\mathbb{R}^d}\int_0^1 (1-s) (y -x)^{\otimes 2}:\nabla^2 
				\Lambda (z(s) )\cdot \frac{\hat{ \rho }_{t_k}(y) \nabla \log \hat{\rho}_{t_k} ( y ) }{ \hat{\rho}_{t} ( x ) }   p ( X _{ t }^{h} = x |X _{ t_{k} }^{h} = y ) d s d y,
\]
and other remainder terms that are easy to control. The term $R_{32}$ has been explained, while the other one 
can be controlled \tcb{using $ \mathbb{E}  \left|\nabla\log\hat{\rho}_{t}(X_{t}^{h} ) \right|^{p_2}   \le C$} obtained in \eqref{eq:integratedlogdensity} (recall the $p_2$ in Assumption \ref{ass:Uniform bound}).
Hence, $\nabla\cdot\bar{r}_t$ can be treated.

Now, consider the divergence of the first term in \eqref{eq:J4aux1}, namely
\begin{multline*}
\nabla\cdot\left(\mathbb{E}\big(X_{ t_{k}}^{h}-X^{h}_t | X^{h}_t=x \big)\cdot \nabla_x \Lambda(t_{k}, x)\right)
=\mathbb{E}\big(X_{ t_{k}}^{h}-X^{h}_t | X^{h}_t=x \big)\cdot \nabla_x (\nabla\cdot\Lambda(t_{k}, x)) \\
+ \sum_{ij}\partial_{x_i}\left(\mathbb{E}\big(X_{ t_{k}}^{h}-X^{h}_t | X^{h}_t=x \big)\right)_j: \partial_{x_j} \Lambda_{i,\cdot}(t_{k}, x)
=M_1+M_2.
\end{multline*}
The term $M_1$ can be dealt with similarly as that in \eqref{eq:J1keyaux} and \eqref{eq:J1firstaux2} so we skip.
Consider $M_2$. The matrix with curiosity can be computed as follows.
\begin{multline}
 \nabla ( \mathbb{E} ( X^{h}_{ t_{k} } - X^{h}_{t} | X^{h}_{t} = x ) )= \int_{\mathbb{R}^d}
\nabla_x  \Big[( y - x ) p (  \hat{ X }_{t } = x  | X_{ t_{k}}^{h} = y) \Big] \frac{ \hat{\rho}_{t_{k}} (y)}{ \hat{\rho}_{t} (x) }dy \\
 -\int_{\mathbb{R}^d} \nabla \log \hat{\rho}_{t}(x) \otimes ( y - x ) p (  \hat{ X }_{t } = x  | X_{ t_{k}}^{h} = y) \frac{ \hat{\rho}_{t_{k}} (y)}{ \hat{\rho}_{t} (x) }	 d y =:  M_{21} + M_{22}.
\end{multline}

For $M_{22}$, we again use the trick to deal with $y-x$ as that in \eqref{eq:J1keyaux} and \eqref{eq:J1firstaux2} to do integration by parts.  Then, we need to estimate $\sqrt{\E (|\nabla\log \hat{\rho}_t(X_t^h)|^2 | \nabla\log \hat{\rho}_{t_{k}}(X_{t_k}^h)|^2)} $. \tcb{Again, this can be controlled using $ \mathbb{E}  \left|\nabla\log\hat{\rho}_{t}(X_{t}^{h} ) \right|^{p_2}   \le C$
in \eqref{eq:integratedlogdensity} at time $t_k$ and $t$(recall the $p_2$ in Assumption \ref{ass:Uniform bound}).}

The $M_{21}$ term is the most singular term. We use the intuition $\nabla_x\sim -\nabla_y$ again.
Direct computation gives the following
\begin{equation}
\begin{split}
& \nabla_{x}  \Big[( y - x ) p (  \hat{ X }_{t } = x  | X_{ t_{k}}^{h} = y) \Big]+\nabla_{y} \Big[( y - x ) p (  X_{t}^{h} = x  | X_{ t_{k}}^{h} = y) \Big] \\
= &\left( (\nabla_x+\nabla_y)\log p (  \hat{ X }_{t } = x  | X_{ t_{k}}^{h} = y)\right)\otimes (y-x) p (X_t^h= x  | X_{t_{k}}^{h} = y)\\
=:&\left(E_1+E_2\right)\otimes(y-x) p (  \hat{ X }_{t } = x  | X_{ t_{k}}^{h} = y),
\end{split}
\end{equation}
with
\[
\begin{split}
& E_1=-\frac{1}{2}\nabla_y\log \det \Lambda(t_k, y)-\nabla_y b(t_k, y)\cdot\Lambda^{-1}(y-x+(t-t_k)b),\\
& E_2=-\left(\frac{(x-y-(t-t_k)b)^T\partial_{y_{\ell}}\Lambda^{-1}(x-y-(t-t_k)b)}{2(t-t_k)}\right)_{\ell=1:d}.
\end{split}
\]
The $E_1$ term and the terms involving $(t-t_k)b$ in $E_2$ are fine, which can be controlled directly using \tcb{Jensen's} inequality and 
H\"older inequality. The essential term in replacing $\nabla_x$ with $-\nabla_y$ is thus 
\[
-\frac{1}{2}(t-t_k)^{-1}((x-y)^T\partial_{y_{\ell}}\Lambda^{-1}(x-y))_{\ell=1:d}\otimes(y-x) p (X_{t}^h = x  | X_{ t_{k}}^{h} = y).
\]
This is similar to $R_{34}$ in \eqref{eq:divLambdaremainder} above. After integration by parts (for one term twice), \tcb{one needs $ \mathbb{E}  \left|\nabla\log\hat{\rho}_{t}(X_{t}^{h} ) \right|^{p_2}   \le C$ in 
\eqref{eq:integratedlogdensity} at time $t_k$ (recall the $p_2$ in Assumption \ref{ass:Uniform bound}).}

Hence, controlling $M_{21}$ term is reduced to the control of the following
\[
\begin{split}
\tilde{M}_{21} &=-\int_{\mathbb{R}^d}\nabla_y  \Big[( y - x ) p (  \hat{ X }_{t } = x  | X_{ t_{k}}^{h} = y) \Big] \frac{ \hat{\rho}_{t_{k}} (y)}{ \hat{\rho}_{t} (x) }dy\\
&=\int_{\mathbb{R}^d}\nabla_y\log \hat{\rho}_{t_{k}} (y)\otimes ( y - x ) p (  \hat{ X }_{t } = x  | X_{ t_{k}}^{h} = y)\frac{ \hat{\rho}_{t_{k}} (y)}{ \hat{\rho}_{t} (x) }dy.
\end{split}
\]
Using again the trick as that in \eqref{eq:J1keyaux} and \eqref{eq:J1firstaux2}, this is reduced to
\[
\int_{\mathbb{R}^d}\nabla_y\log \hat{\rho}_{t_{k}} (y)\otimes [-(t-t_k)\Lambda(t_k, y)\cdot\nabla_y p (  \hat{ X }_{t } = x  | X_{ t_{k}}^{h} = y)]\frac{ \hat{\rho}_{t_{k}} (y)}{ \hat{\rho}_{t} (x) }dy.
\]
Integration by parts again, we can control this main term eventually, \tcb{with the help of the control for
$\E |\nabla^2\log\hat{\rho}_{t_k}(X_{t_k}^h)|^{p_3}$  and $\E |\nabla\log \hat{\rho}_{t_{k}} (X_{t_k}^h)|^{p_2}$  in \eqref{eq:integratedlogdensity} (recall the $p_2$, $p_3$ in Assumption \ref{ass:Uniform bound}).}

\vskip 0.2 in

\noindent {\bf Step 4. The overall estimate}

Hence, one obtains that for $ t \in [ t_{k}, t_{k+1} ] $,
\begin{equation}
\frac{d}{d t} \cH(\hat{\rho}_t | \rho_t)
	\leq C ( t - t_{k} )^2-\frac{\kappa}{4}\int_{\mathbb{R}^d} \hat{\rho}_t | \nabla \log \frac{\hat{\rho}_t}{\rho_t}|^2 d x.
\end{equation}
It follows that
$\cH(\hat{\rho}_t | \rho_t) \leq C h^2$ as desired for all $t\le T$.
 \end{proof}

\section{\tcb{Numerical Experiments}}\label{sec:experiments}
\tcb{This section provides numerical evidence for the second-order convergence in relative entropy of the Euler-Maruyama scheme, consistent with Theorem~\ref{thm:errentropy}, through two benchmark test cases.
First, geometric Brownian motion is employed as a canonical stochastic differential equation to verify the theoretical convergence rate, representing a fundamental class of diffusion processes. Second, a modified Ginzburg-Landau model examines the convergence behavior of the algorithm in higher-dimensional settings with coupled multiplicative noise.  
All densities are approximated via kernel density estimation (KDE) with extensive Monte Carlo simulations ($10^6$ paths per configuration), using the bandwidth selection method proposed by Silverman~\cite{silverman2018density} and numerical stabilization ($\epsilon=10^{-10}$). Convergence metrics are computed over logarithmically spaced grids to ensure accuracy across different density regimes.}

\subsection{\tcb{Geometric Brownian Motion (1D Model)}}\label{subsec:gbm_experiment}
\tcb{We analyze the geometric Brownian motion process
\begin{equation}\label{eq:gbm}
    dX_t = \alpha X_t dt + \sigma X_t dW_t, \quad X_0=1,
\end{equation}
with $\alpha=0.6$ and $\sigma=0.1$, chosen for its well-understood dynamical properties and broad applicability in stochastic modeling. The reference solution is generated using the Milstein scheme:
\begin{equation}\label{eq:milstein}
    X_{n+1} = X_n + \alpha X_n h_{\text{ref}} + \sigma X_n\Delta W_n + \frac{\sigma^2}{2}X_n(\Delta W_n^2 - h_{\text{ref}})
\end{equation}
at $h_{\text{ref}}=2^{-12}$ with $M=10^6$ independent paths.}

\tcb{The density $\rho_T(x)$ is estimated via Gaussian kernel density estimation:
\begin{equation}\label{eq:kde}
    \rho(x) = \frac{1}{M\lambda_n}\sum_{j=1}^M \frac{1}{\sqrt{2\pi}}e^{-\frac{(x-X_T^{(j)})^2}{2\lambda_n^2}}
\end{equation}
}
\tcb{where $\lambda_n=0.05$ is the Silverman-optimized bandwidth and $X_T^{(j)}$ denotes terminal values. Evaluations use a 1000-point logarithmic grid $x_i\in[10^{-3},10^1]$ with spacing $\Delta x_i =0.0019$.}

\tcb{For convergence analysis, the Euler-Maruyama scheme:
\begin{equation}\label{eq:euler}
    X_{n+1} = X_n + \alpha X_n h_k + \sigma X_n\Delta W_n
\end{equation}
is tested at $h_k=2^{-k}$ ($k=3,\dots,7$). Relative entropy is computed as:
\begin{equation}\label{eq:entropy}
    \mathcal{H}(\hat{\rho}\|\rho) = \sum_{i=1}^{1000} \max(\hat{\rho}(x_i),\epsilon) \log\left(\frac{\max(\hat{\rho}(x_i),\epsilon)}{\rho_T(x_i)}\right)\Delta x_i
\end{equation}
with $\epsilon=10^{-10}$ regularization.}

\begin{figure}[!h]
	\centering
	\includegraphics[
	width=0.7\linewidth]{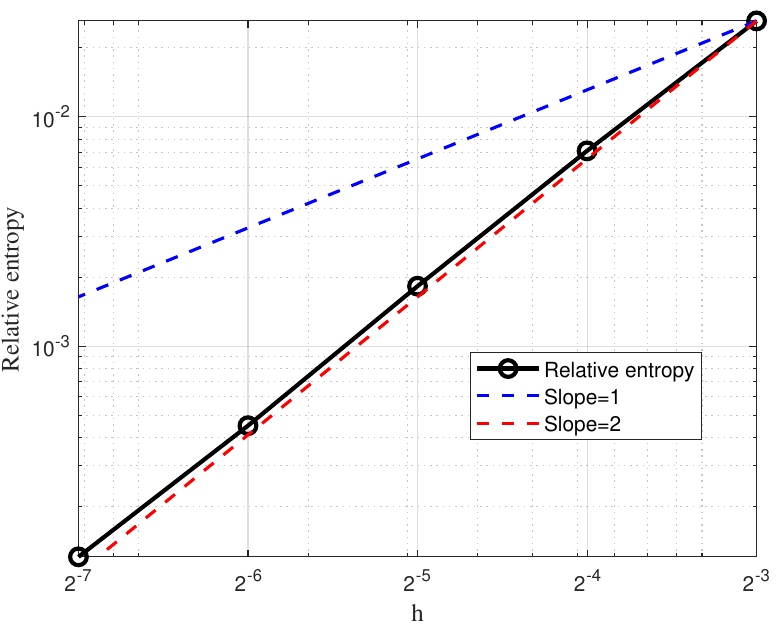}
	\caption{\tcb{
		Log-log plot of the relative entropy $\mathcal{H}(\rho_h | \rho_{\mathrm{ref}})$ versus step size $h$, where $\rho_{\mathrm{ref}}$ is computed at $h_{\mathrm{ref}}=2^{-12}$. The empirical convergence rate (black circles with solid line) closely aligns with the theoretical second-order reference (red dashed line), while showing clear deviation from the first-order behavior (blue dashed line).}
	}
	\label{fig:entropy_convergence}
\end{figure}

\tcb{The numerical results validate the second-order convergence predicted by Theorem~\ref{thm:errentropy}, demonstrating a significant improvement over existing first-order guarantees for systems with multiplicative noise. As shown in Figure~\ref{fig:entropy_convergence}, the log-log plot of relative entropy exhibits clear $\mathcal{O}(h^2)$ scaling, confirming the theoretical convergence rate. This represents a twofold convergence rate improvement compared to standard first-order methods. Also,
Figure~\ref{fig:entropy_convergence} shows that at $h=2^{-7}$, 
the relative entropy reaches $\mathcal{O}(10^{-4})$ magnitudes, 
indicating good convergence performance of the Euler method 
in relative entropy measurement.}

\subsection{\tcb{A modified Ginzburg-Landau System}}
\tcb{The Ginzburg-Landau (GL) equation, originally developed in superconductivity theory, describes phase transitions in complex systems. Its stochastic version for a two-component field $(X_t,Y_t)$ takes the form:
\begin{equation}\label{eq:gl_original}
	\begin{cases}
		dX_t = (\alpha X_t - \gamma Y_t - \beta X_t^3) dt + \sigma X_t dW_t^1, \\
		dY_t = (\alpha Y_t + \gamma X_t - \beta Y_t^3) dt + \sigma Y_t dW_t^2.
	\end{cases}
\end{equation}
Where $\alpha$ controls linear stability, $\gamma$ determines rotational coupling, $\beta>0$ governs nonlinear saturation, $\sigma$ is the noise intensity, and where $W_t^1$ and $W_t^2$ are two independent standard Brownian motions. The cubic nonlinearity $-X_t^3$ ensures bounded solutions but introduces non-global Lipschitz continuity, potentially causing numerical instabilities (as demonstrated by M. Hutzenthaler et al. \cite{hutzenthaler2011strong}, where the Euler scheme diverges for non-globally Lipschitz coefficients).
To ensure numerical tractability while preserving physical characteristics, we employ a tamed version through nonlinearity regularization:
\begin{equation}\label{eq:gl_tamed}
	\left\{
	\begin{aligned}
		dX_t &= \left( \alpha X_t - \gamma Y_t - \beta \frac{X_t^3}{1 + X_t^2} \right) dt + \sigma X_t dW_t^1 ,\\
		dY_t &= \left( \alpha Y_t + \gamma X_t - \beta \frac{Y_t^3}{1 + Y_t^2} \right) dt + \sigma Y_t dW_t^2.
	\end{aligned}
	\right.
\end{equation}
The taming term $\frac{u^3}{1+u^2}$ is globally Lipschitz continuous and exhibits linear growth.}

\tcb{Numerical experiments were conducted on the two-dimensional Ginzburg-Landau system with parameters $\alpha=0.5$, $\beta=1$, $\gamma=3$, and $\sigma=0.5$. Numerical solutions were generated via the Euler-Maruyama scheme with time steps $h_k=2^{-k}$ ($k=5,...,9$), while the reference solution was computed using the Milstein scheme at $h_{\text{ref}}=2^{-14}$, both employing $10^6$ Monte Carlo samples. The kernel density estimation in two dimensions exhibits substantially higher computational demands: the bandwidth selection extends from a scalar $\lambda_n=0.05$ in 1D to a covariance matrix (diagonal bandwidth $\lambda_{ii}=0.15$) in 2D, the grid points increase from 1,000 to 40,000 resulting in 150-fold higher memory requirements, and the pseudo-count threshold must be relaxed to $\epsilon=10^{-10}$ to maintain numerical stability.}

\begin{figure}[!ht]
	\centering
	\includegraphics[width=0.7\linewidth]{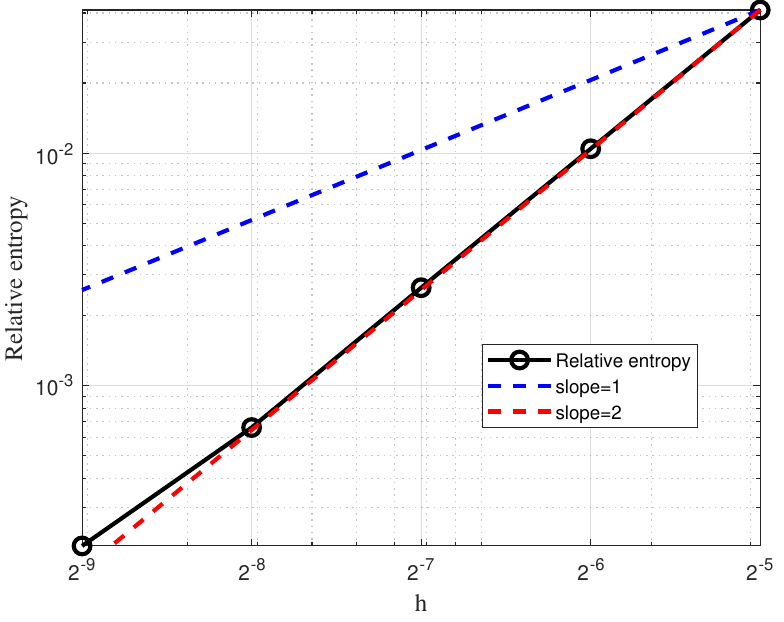}
	\caption{
		\tcb{Estimate for relative entropy in the modified Ginzburg-Landau system.
Log-log plot of the relative entropy $\mathcal{H}(\rho_{h} \| \rho_{\mathrm{ref}})$ versus time step $h$, 
where the reference distribution $\rho_{\mathrm{ref}}$ is computed at $h_{\mathrm{ref}} = 2^{-14}$.
The empirical convergence rate (black circles with solid line) closely aligns with the theoretical second-order reference (red dashed line), while showing clear deviation from the first-order behavior (blue dashed line).}
}
	\label{fig:GL_convergence}
\end{figure}

\tcb{Numerical analysis confirms second-order convergence in relative entropy for the Euler method applied to two-dimensional stochastic systems. The observed $\mathcal{O}(h^2)$ decay fully aligns with the theoretical prediction of Theorem~\ref{thm:errentropy}, with empirical results showing clear deviation from first-order behavior across all tested step sizes. Moreover, Figure~\ref{fig:GL_convergence} verifies these convergence characteristics at practical resolutions: applying $h=2^{-9}$ reduces relative entropy to $\mathcal{O}(10^{-4})$ magnitudes, indicating its reasonably good convergence behavior in entropy computation for the tested two-dimensional system.}

\section{Conclusion}
In this paper, we consider the numerical discretization of SDEs with multiplicative noise.
Our result focuses on two key aspects: (1) the $L^p$-upper bounds for derivatives of the logarithmic numerical density, (2) the sharp convergence analysis of Euler's scheme under the relative entropy.
In the first part, we present estimates for the first and second order derivatives of the logarithmic numerical density. mainly using the Malliavin calculus to derive expressions of the derivatives of the logarithmic Green's function and to obtain an estimate for the inverse Malliavin matrix.
In the second part, in terms of the relative entropy error, we obtain a bound that has a second-order accuracy in time step, which then naturally leads to first-order error bounds under the total variation distance and Wasserstein distances. 
\tcb{We finally discuss here possible extensions of our analysis technique to other numerical schemes for SDEs. Currently, our proof framework cannot be applied to the implicit Euler method, but we think it is possible to obtain similar results by finding some suitable continuous time interpolation.
Other} possible improvements regarding the current results include: extending the current convergence analysis to a uniform-in-time one, and relaxing the noise to the degenerate case (for instance, the underdamped Langevin equation). We leave these as possible future work.

\section*{Acknowlegement}

This work was financially supported by the National Key R\&D Program of China, Project Number 2021YFA1002800.
The work of L. Li was partially supported by NSFC 12371400 and 12031013,  Shanghai Municipal Science and Technology Major Project 2021SHZDZX0102, and Shanghai Science and Technology Commission (Grant No. 21JC1403700, 21JC1402900).
M. Wang's work was partially supported by the Postdoctoral Fellowship Program of CPSF under Grant Number GZC20241020. \tcb{The authors would like to thank the editors and the anonymous reviewers for helpful comments and suggestions.}

\appendix

\section{Pointwise estimate for 1D time continuous SDE}\label{app:pointwise}

In this section, we provide a pointwise estimate for $R$ and $R_2$ defined in \eqref{eq:firstordercommn} and
\eqref{eq:secondordercommn}. This estimate can naturally give the well-known Li-Yau type pointwise estimate for the solution to Fokker-Planck equation, which can be proved by some classical PDE methods as well (e.g. the Bernstein's method). We will also point out why the argument fails for numerical solutions, which led us to use the integrated version as in the main part of the paper. 

Similarly, the process to consider for the Green's function of the SDE is given by
\begin{gather}
Y_t=y+\int_a^t b(s, Y_s)\,ds+\int_a^t \sigma(s, Y_s)\,dW_s,
\end{gather}
where $a$ now is an arbitrary point in $[0, T)$ and $a\le t\le T$. Here, $b(\cdot, \cdot): \R_+\times \R\to \R$
and $\sigma(\cdot,\cdot): \R_+\times \R\to \R$ and $W_s$ is a 1D standard Brownian motion. 
\begin{remark}
For $d=1$, the more general case is given by $\sigma(s, y)\in \R^{m\times 1}$ and $W$ is $m$-dimensional Brownian motion.
Studying the law for this general case is in fact equivalent to the case here since $\sum_i \lambda_i z_i\sim \sqrt{\sum_i \lambda_i^2}\tilde{z}$
for independent Gaussian variables $z_i$, and $\tilde{z}$ is also a Gaussian variable.
\end{remark}

For a differentiable function $ f $, we denote
$$
f'_{1} ( t,x) := \frac{\partial f }{\partial t} ( t,x),
\quad
f'_{2} ( t,x) := \frac{\partial f }{\partial x} ( t,x).
$$

Instead of polynomial growth, here we require a bit stronger conditions. Note that we also require the boundedness of the drift (for $\alpha = 0$ below).

\begin{assumption}\label{ass:condition-drift-diffusion2}

\begin{enumerate}[(a)]
\item The drift coefficient $ b \in C^3
( \R_+ \times \mathbb{R}^{d} ; \mathbb{R}^d)$ satisfies for any multi-index $\alpha$ with $0\le |\alpha|\le 3$ that
\begin{gather}
\sup_{t,x}| \partial_x^{\alpha} b(t,x)| <\infty.
\end{gather}

\item The diffusion coefficient $\sigma\in C^3(\R_+\times \R^d; \R^{d\times m})$ satisfies for any multi-index $\alpha$ with $0\le |\alpha|\le 3$ that
\begin{gather}
\sup_{t,x}|\partial_x^{\alpha}\sigma(t, x)|<\infty,\quad \sup_{t,x}|\partial_t\sigma(t, x)|<\infty.
\end{gather}
\end{enumerate}
\end{assumption}

\begin{proposition}\label{pro:greenfirstorder1}
Suppose Assumption \ref{ass:condition-drift-diffusion2} holds. Consider \eqref{eq:firstordercommn} and \eqref{eq:secondordercommn}, which is written specifically for 1D SDE case as
\begin{gather}
\begin{split}
& R(t, x, a, y) = \frac{\partial}{\partial x} \log p(t, x, a, y) +\frac{\partial}{\partial y} \log p(t, x, a, y),\\
& R_2(t, x, a, y) = \left(\frac{\partial}{\partial x}+\frac{\partial}{\partial y}\right)^2 \log p(t, x, a, y).
\end{split}
\end{gather} 
Then we have
\begin{equation}
| R(t,x,a,y) | \leq C \left( 1 + \frac{|x - y |^2}{t-a} \right), \quad |R_2(t, x, a, y) | \leq C \left( 1 + \frac{ | x - y |^4 }{ (t-a)^2 } \right).
\end{equation}
\end{proposition}

Before the proof, we note some formulas for the stochastic Jacobian
\begin{gather}
J_t=\frac{\partial Y_t}{\partial y}
\end{gather}
and the Malliavin derivative $DY_t$.
\begin{lemma}
Consider the so-called feedback rate 
\begin{gather}
\lambda(t, x)=b_2'-\frac{\sigma_2'}{\sigma}b-\frac{\sigma_1'}{\sigma}-\frac{1}{2}\sigma_2'' \sigma.
\end{gather}
The Jacobian $J$ is given by
\begin{gather}
J_t=\frac{\sigma(t, Y_t)}{\sigma(a, y)}\exp\left(\int_a^t \lambda(s, Y_s)\,ds\right).
\end{gather}
and for $r\in [a, T]$, one has
\begin{gather}
D_rY_t=
\begin{cases}
0 & t<r \\
\sigma(t, Y_t)\exp(\int_r^t\lambda(s, Y_s))\,ds & t\ge r.
\end{cases}
\end{gather}
\end{lemma}
Note that the stochastic Jacobian satisfies
\begin{gather}
dJ=b_2' J dt+\sigma_2' J dW,\quad J_a=1.
\end{gather}
The idea is to use the equation for $dY$ to kill the It\^o's integral term.  Since $d\sigma(t, Y_t)$ would give a term like $\sigma_2' \sigma dW$, then it is natural to expect that $J_t=C\sigma(t, Y_t)\exp(z(t))$ where $z$ has zero quadratic variation. Inserting this into the equation for $J$, one can obtain the result for $J_t$.
The Malliavin derivative $D_rY_t$ is similar. In fact, for $t\ge r$, the equation is given by
\begin{gather}
D_rY_t=\sigma(r, Y_r)+\int_r^t b_2'(s, Y_s) D_rY_s\,ds
+\int_r^t \sigma_2'(s, Y_s) D_rY_s\,dW_s.
\end{gather}

\begin{proof}[Proof of Proposition \ref{pro:greenfirstorder1}]
Different from the numerical density, we choose the covering field to be
\begin{gather}
u=\frac{v}{\langle D Y_t, v \rangle_H},\quad v_s=\mathbb{E}\left[D_s Y_t | \hat{\mathscr{F}}_s\right].
\end{gather}
 The benefit is that $v_s$ is adapted and Clark-Ocone formula \cite{nualart2018introduction} gives
\[
\delta(v)=\int_a^t v_s\,dW_s=Y_t-\E Y_t.
\]
The main benefit in the time-continuous case is that $D_rY_t$ is bounded below by a positive number uniformly almost surely
for $r\le t$. Hence, 
\begin{gather}\label{eq:Mmatrixpointwise}
\alpha(T) (t-a)\le  \langle DY_t, v\rangle_H \le \beta(T)(t-a),
 \end{gather}
 for some deterministic $\alpha(T)>0$ and $\beta(T)>0$.

By the expression, one has
\begin{gather}
R(t, x, a, y)=\mathbb{E} ( \delta ( u ( J_t - 1 ) ) | Y_t = x )
\end{gather}
Note that
\[
\begin{split}
\delta ( u ( J_t - 1 ) )&=\delta(u)(J_t-1)+\langle u, DJ_t\rangle\\
&=\frac{Y_t-\E Y_t}{ \langle DY_t, v\rangle_H}(J_t-1)
-\left\langle v, \frac{D \langle DY_t, v\rangle_H}{\langle DY_t, v\rangle_H^2}\right\rangle_H(J_t-1)
+\langle u, DJ_t\rangle.
\end{split}
\]
Now, we use the explicit formulas for $J_t$ and $DY_t$, which is beneficial since $Y_t$ in these formulas can be replaced by $x$, conditioning on $Y_t=x$. The first term, which is the main term is reduced to
\begin{gather*}
\E\left[\frac{Y_t-\E Y_t}{ \langle DY_t, v\rangle_H}(J_t-1)| Y_t=x \right]=(x-\E Y_t)\E\left[\frac{\sigma(t, x)\exp(\int_a^t\lambda(s, Y_s)\,ds)-\sigma(a, y)}{\sigma(a, y) \langle DY_t, v\rangle_H}| Y_t=x\right]
\end{gather*}
Using the assumptions, it is not hard to find the desired pointwise bound.
Other terms are relatively straightforward, though more tedious.

For $R_2$, we have by Lemma \ref{lmm:greenexpression} that
\begin{gather*}
R_2(t, x, a, y)
= \mathbb{E} \Big[  \delta\big( u \delta ( u  (\partial_y \hat{X}_t-1)^2  ) \big) + \delta( u \partial_y^2\hat{X}_t )   | \hat{X}_t = x   \Big]
-R^2(t, x, a, y).
\end{gather*}
The estimate is similar (we need pointwise estimate of more terms, like $|DD Y_t|\le C$, etc, but these are straightforward) and we skip the details.
\end{proof}

Using the pointwise estimates of $R$ and $R_2$, we can get the following pointwise estimate of the logarithmic density (which is the well-known Li-Yau type pointwise estimate).
\begin{proposition}\label{lem:one-derive-tail}
	Suppose the coefficients satisfy assumption \ref{ass:condition-drift-diffusion} and the initial density $\rho_0(x)$ satisfies the that
	\[
	C_1\exp(-\alpha_1 |x|^2)\le \rho_0(x)\le C_2\exp(-\alpha_2 |x|^2)
	\] 
	and
	\[
	\left|\frac{\partial}{\partial x}\log\rho_0(x)\right|\le C(1+|x|^2), \quad \left|\frac{\partial^2}{\partial x^2} \log\rho_0(x)\right|\le C(1+|x|^4),
	\]
	then the density $\rho_t(x)$ for $X_t$, which is the density of the SDE satisfies
\begin{equation}\label{eq:rho}
	C_1(T)\exp(-r_3|x|^2)\le \rho_t(x)\le C_2(T) \exp(-r_4|x|^2),
\end{equation}
and
\begin{equation}\label{eq:log-rho}
	\left|\frac{\partial}{\partial x} \log \rho_t(x)\right|\le C(T)(1+|x|^2), \quad \left|\frac{\partial^{2}}{\partial x^2} \log\rho_t(x)\right| \le C(T)(1+|x|^4).
\end{equation}
\end{proposition}

\begin{proof}
The Green's function (transition probability) satisfies (see, for example, \cite{sheu1991some})
\begin{gather}
 C (t-a)^{-d/2}\exp\left(-\frac{\alpha |x-y|^2}{t-a}\right) \le p(t, x, a, y)\le C' (t-a)^{-d/2}\exp\left(-\frac{\alpha' |x-y|^2}{t-a}\right).
\end{gather}
Then, the Gaussian tail estimate \eqref{eq:rho} follows.

Next, we take the argument for the first-order derivative as the example. One has
\begin{equation}
\begin{split}
\tfrac{\partial}{\partial x} \log \rho_{ t }(x)&=\frac{1}{\rho_t(x)}\int_{\mathbb{R}}\tfrac{\partial}{\partial x}  p(t, x, 0, y) \rho_0(y)  d y\\
&=\frac{1}{\rho_t(x)}\int_{\mathbb{R}}(\partial_y\log\rho_0(y)+R(t,x,0,y)) p_t(x|y) \rho_0(y)  d y.
\end{split}
\end{equation}

	Since 
	\[
	|\log \rho_0(y)|\le C(1+|y|^2)
	\le C(1+|x|^2)+C|x-y|^2,
	\]
	we find that
	\[
	|\partial_y\log\rho_0(y)+R(t,x,0,y)|
	\le C(1+|x|^2)+C\frac{|x-y|^2}{t}.
	\]
	Hence, it suffices to estimate
	\[
	H_1:=\int_{\mathbb{R}}\frac{|x-y|^2}{t}p_t(x|y) \rho_0(y)  d y.
	\]
	Let $\bar{R}:= A \sqrt{t} \max \left( 1 ,  \sqrt{\left|\log \rho_t(x)\right|}\right)$. One then has
	\begin{equation}
		\begin{split}
			\int_{ |x - y | \leq \bar{R} }p_t(x|y)  \frac{|x - y |^{2}}{t}\rho_0(y)  d y
			\leq \frac{\bar{R}^2}{t}\int_{ \mathbb{R} }p_t(x|y)\rho_0(y)  d y
			\le C(1+|x|^2)\rho_t(x),
		\end{split}
	\end{equation}
	where we used $|\log\rho_t(x)|\le C(1+|x|^2)$.
	On the other hand,
	\begin{equation}
		\begin{split}
			\int_{ |x - y | \geq \bar{R}}
			\cdots  d y
			\leq &C\int_{ |x - y | \geq \bar{R} }
			t^{-\frac{1}{2}} 
			e^{-\frac{ c | x - y |^2}{t}}
			\frac{|x - y |^2}{t}\rho_0(y)  d y\\
			\leq &  C \int_{ |x - y | \geq \bar{R} }
			t^{-\frac{1}{2}} 
			e^{-\frac{ c | x - y |^2}{2t}}
			e^{-\frac{ c A^2 t | \log \rho_t(x)| }{4t}}
			\left(e^{-\frac{ c | x - y |^2}{4t}}\frac{|x-y|^2}{t}\right)\rho_0(y)  d y\\ 
			\leq &
			C e^{-\frac{ c A^2  | \log \rho_t(x)| }{4}}\int_{ \mathbb{R} }
			t^{-\frac{1}{2}} 
			e^{-\frac{ c | x - y |^2}{2t}} \rho_0(y)  d y\leq C  \rho_t(x),
		\end{split}
	\end{equation}
	where we have chosen $A = \sqrt{\frac{4}{c}}$. Hence,
	\[
	\frac{H_1}{\rho_t(x)}\le C(1+|x|^2),
	\]
	giving the desired control. Hence, the estimate for $|\partial_x\log\rho_t|$ follows.

	The second-order derivative is pretty much similar and we omit the details.
	\end{proof}

Lastly, let us conclude this appendix with a remark.
The pointwise estimate here heavily relies on \eqref{eq:Mmatrixpointwise}.
For numerical density, one can see clearly that $D_r Y_t^h$ is not almost surely bounded below by a positive constant, which makes the argument here fail. Alternatively,  we choose $v=DY_t^h$ instead of $\E(D_sY_t^h | \hat{\mathscr{F}}_s)$.
Proving the pointwise estimate of the Malliavin matrix \eqref{eq:discreteMmatrix} directly is also difficult, and this is why we choose the integrated version eventually.
Besides, applying the PDE approach (like the Bernstein method) to the numerical density seems also challenging as
we do not have estimates for $\hat{b}_t$ and $\hat{\Lambda}_t$ in Lemma \ref{lem:Fokker-eq}.

\bibliographystyle{plain}

\bibliography{refer}

\end{document}